\documentclass[11pt]{amsart}%%%%%,leqno
\usepackage{amsmath,amssymb, graphicx, amscd,latexsym,comment,color}
\makeatletter
\newtheorem{Theorem}{Theorem}%[section]
\newtheorem{Lemma}[Theorem]{Lemma}
\newtheorem{Corollary}[Theorem]{Corollary}

\newtheorem{Sublemma}[Theorem]{Sublemma}
\newtheorem{Definition}[Theorem]{Definition}
\newtheorem{Remark}[Theorem]{Remark}

\newtheorem{Example}[Theorem]{Example}

\newtheorem{Assertion}[Theorem]{Assertion}

\newcommand{\eps}{\varepsilon}

\newcommand\la{\lambda}
\newcommand\vphi{\varphi}

\newcommand\al{\alpha}

\newcommand\be{\beta}
\newcommand\Si{\Sigma}
\newcommand\ga{\gamma}
\newcommand\Ga{\Gamma}
\newcommand\de{\delta}

\newcommand\BC{ {\mathbb C}}

\newcommand\bfz{\mbox {\bf  z}}

\newcommand\bfb{\mbox {\bf  b}}

\newcommand\nl{\newline}

\newcommand\order{\mbox{\rm{order}\/}}

\newcommand\ord{{\rm{ord}\/}}
\newcommand\Int{\rm{Int}\/}

\newcommand\modulo{\rm{modulo}\/}

\newcommand\inv{^{-1}}

\def\mapright#1{\smash{\mathop{\longrightarrow}\limits^{{#1}}}}%\def\maprightt#1#2{\smash{\mathop{\longrightarrow}\limits^{#1}}}
\def\mapdown#1{\Big\downarrow\rlap{$\vcenter{\hbox{$#1$}}$}}

\def\inv{^{-1}}

\begin{document}
\title[On the  Milnor fibration for $f(\mathbf z)\bar g(\mathbf z)$ II
%of plane curves and non-reduced degeneration
%{ \today}
]
{On the Milnor fibration for $f(\mathbf z)\bar g(\mathbf z)$ II
}

\author
%Normally smooth divisors  { \today}
[M. Oka ]
{Mutsuo Oka }
\address{\vtop{
\hbox{Department of Mathematics}
\hbox{Tokyo  University of Science}
\hbox{1-3, Kagurazaka, Shinjuku-ku}
\hbox{Tokyo 162-8601}
\hbox{\rm{E-mail}: {\rm oka@rs.tus.ac.jp}}
}}

%\today
%\thanks{}
\keywords {locally tame, non-degenerate, toric multiplicity condition}
\subjclass[2000]{14J70,14J17, 32S25}

\begin{abstract}
We consider a mixed function of type $H(\mathbf z,\bar{\mathbf z})=f(\mathbf z)\bar g({\mathbf z})$ where $f$ and $g$ are %Newton non-degenerate
  holomorphic functions which  are non-degenerate  with respect to the Newton boundaries.  
  % of $f$ and $g$.
  % but we do not assume their convenience.
We assume  also that  the variety $f=g=0$ is  a non-degenerate  complete intersection variety.
 In our previous paper, we considered the case 
that $f,g$  are convenient so that they have isolated singularities.  In this paper we do not assume the convenience  of $f$ and $g$.
 In non-convenient case,  two hypersurfaces  may have non-isolated singularities at the origin.
% I introduce the multi %in  the sense of \cite{Okabook}. 
 We will show that $H$ has  still both  a tubular and  a spherical Milnor fibrations
  under the local tame non-degeneracy and the toric multiplicity condition.  
  We prove also the equivalence of two fibrations.
 % In the last  section, we  will prove the tubular Milnor fibration theorem replacing the non-deneneracy condition with a milder condition.
\end{abstract}
\maketitle

\maketitle
  \noindent
\section{Locally tame non-degenerate complete intersection pair}
\subsection{Introduction} Let $f(\mathbf z)$ and $g(\mathbf z)$ be holomorphic functions vanishing at the origin. For $h(\mathbf z):=f(\mathbf z)g(\mathbf z)$, there exists a tubular Milnor fibration $h:E(r,\de)^*\to D_\de^*$ or  a spherical Milnor fibration
$h/|h|: S_r\setminus K_r\to S^1$ for small $r$ and $\de\ll r$ (\cite{Milnor,Hamm-Le1}.
Here $E(r,\de)^*:=\{\mathbf z\in B_r^{2n}\,|\, 0\ne |f(\mathbf z)\le \de\}$ and $K_r:=f\inv(0)\cap S_r^{2n-1}$.
We consider the mixed function 
$H(\mathbf z,\bar{\mathbf z}):=f(\mathbf z)\bar g(\mathbf z)$ and the existence problem of its Milnor fibration. 
The link of $H$ is the same as the complex link given by $h(\mathbf z)$ but the fibration structure along the link of $g=0$ is conversely oriented.
 It turns out that such a fibration does not exist for an arbitrary pair. This problem has been studied by several authors but there are not yet satisfactory results (\cite{Pichon-Seade0,Pichon-Seade1,Pichon-Seade2, Param-Tibar-revised}).
For a non-degenerate mixed function, it is known that the Milnor fibration exists (\cite{OkaMix}). However for $n\ge 3$, $H$ can not be non-degenerate as $f=g=0$ is a non-isolated singular locus for $H$. In our previous paper \cite{fg-bar}, we have shown the existence of Milnor fibrations for $H$ under the assumption that $f,g$ are convenient non-degenerate functions, satisfying the multiplicity condition. A convenient non-degenerate function $f$  has an isolated singularity at the origin. In this paper, we consider the same problem without assuming the convenience.
That is, we consider the case  that $f=0$  or $g=0$ may have non-isolated singularity at the origin.
%%%%%
\subsection{Vanishing coordinate subspaces and locally tameness}
Let $f(\bfz)$ be a  holomorphic function of
$n$ complex variables $z_1,\dots, z_n$ which vanishes at the origin. Consider 
a coordinate subspace $\mathbb C^I:=\{(z_1,\dots,z_n)\in \mathbb C^n\,|\, z_j=0,\,j\notin I\}$ where $I\subset \{1,2,\dots,n\}$.
$\mathbb C^I$ is called {\em a vanishing coordinate subspace} of $f$ if the restriction of $f$ to 
 $\mathbb C^I$
%$\mathbb C^I:=\{(z_i)\,|\, z_i=0,\,j\notin I\}$
 is identically zero. The restriction of $f$ is denoted as $f^I$.
We denote the set of vanishing subspaces of $f$ (respectively of $g$) by  $\mathcal V_f$  (resp. by $\mathcal V_g$).
Let $P=(p_1,\dots, p_n)$ be a semi-positive weight vector. We put  $I(P):=\{i\,|\, p_i=0\}$.
Take a vanishing coordinate subspace $\mathbb C^I$ and 
take an arbitrary   semi-positive weight vector $P=(p_1,\dots, p_n)$   such that $I(P)=I$.
Then the face function
 $f_P$ is a weighted homogeneous function of  the variables 
$(z_j)_{j\notin  I}$ with a positive degree $d(P;f)$ with respect to the weight vector $P$.

Recall that  $f$ is non-degenerate if for any strictly positive weight vector $P$ (i.e.,  $I(P)=\emptyset$),
$f_P:\mathbb C^{*n}\to \mathbb C$ has no critical points (\cite{OkaPrincipal}).
We say that the function $f$  (or the hypersurface $V(f):=f\inv(0)$)  is {\em  locally tame and non-degenerate} if 
it is non-degenerate and for any vanishing coordinate subspace $\mathbb C^I$, there exists a positive number $r_I$ such that for any weight vector  $P$ with $I(P)=I$,
%there exists a positive number $r_I$
 %in addition of Newton non-degeneracy,
% so that 
$f_P$ is a non-degenerate function of $(z_j)_{j\notin I}$ with the other variables   $(z_i)_{i\in I}\in \mathbb C^{*I}$ being fixed in the ball $\sum_{i\in I}|z_i|^2\le r_I$ (\cite{OkaAf,EO}).  Put $V(f)^\sharp:=\cup_{\mathbb C^I\notin \mathcal V_f}V(f^I)\cap \mathbb C^{*I}$.
Recall that $V(f)^\sharp$ is smooth near the origin  (Lemma (2.2), \cite{Okabook}).

%We say that 
%$f,g$ satisfies {\em the disjointness of vanishing coordinate subspaces} if $\mathcal V_f\cap\mathcal V_g=\emptyset$.
%
%Consider now   a pair of holomorphic functions $f(\mathbf z), g(\mathbf z)$ which vanish at the origin. 
For the pair of function $\{f,g\}$, consider the following conditions.
\begin{enumerate}
\item The hypersurfaces $V(f)=f\inv(0), V(g)=g\inv(0)$ are locally  tame and non-degenerate.
\item The variety
$V(f,g)=\{f=g=0\}$ is a locally tame non-degenerate complete intersection variety. Namely (2-a)
for  any strictly positive weight vector $P$, the variety $\{\mathbf z\in \mathbb C^{*n}\,|\, f_P(\mathbf z)=g_P(\mathbf z)=0\}$ is a smooth complete intersection variety. (2-b) For any common vanishing coordinate subspace $\mathbb C^I$, there exists a positive number $r_I$ such that for any weight vector  $P$ with $I(P)=I$,
$\{f_P=g_P=0\}$ is a non-degenerate complete intersection variety in $\mathbb C^{*J}$ with  $J=I^c$ and $\mathbf z_I \in \mathbb C^{*I}$ is fixed  in the ball $\sum_{i\in I}|z_i|^2\le r_I$.
%\item $\{f,g\}$ satisfies the disjointness of vanishing subspaces, i.e.  $\mathcal V_f\cap\mathcal V_g=\emptyset$.
%Here $I=I(P)$.

\end{enumerate}
We say $\{f,g\}$ is {\em a  locally tame non-degenerate pair}  if it satisfies only (1) and (2).
The pair $\{f,g\}$  
is  {\em a disjoint  locally tame non-degenerate complete intersection pair} if it satisfies (1), (2-a) and 
(3) $\{f,g\}$ satisfies the disjointness of vanishing subspaces, i.e.  $\mathcal V_f\cap\mathcal V_g=\emptyset$.

%We say that 
%$f,g$ satisfies {\em the disjointness of vanishing coordinate subspaces} if $\mathcal V_f\cap\mathcal V_g=\emptyset$.

Note that if  $\mathbb C^{I}\in \mathcal V_f\setminus\mathcal V_g$, 
there exists a positive number $r_I$  such that  for any semi-positive weight vector $P$ with $I(P)=I$,  $g_P=g^I$ and %the following conditions are satisfied.
$f_P=g_P=0$ is a non-degenerate complete intersection variety.
\begin{Remark}
If $f$ is locally tame and  non-degenerate and if $\mathbb C^I$ is not a vanishing coordinate subspace for $f$, $f^I$ is also locally tame and  non-degenerate as a function on $\mathbb C^I$. %Recall that the restriction $f|_{\mathbb C^I}$ is denoted as $f^I$.
See the argument in Proposition (1.5), Chapter III \cite{Okabook}. Locally tameness has been defined for mixed functions (Definition 2.7, \cite{EO}). If a holomorphic function $f(\mathbf z)$ is locally tame, it is also locally tame as a mixed function.

\end{Remark}

In this paper, we consider the existence problem of the Milnor fibration of $H(\mathbf z,\bar{\mathbf z})=f(\mathbf z)\bar g(\mathbf z)$ under the assumption that $f,g$  is a locally tame non-degenerate complete intersection pair.
For  the further detail about a non-degenerate complete intersection variety, see  \cite{OkaPrincipal,Okabook}.
\subsection{Examples }
%of
%locally tame non-degenerate functions and non-degenerate complete intersection and locally tame pair}
Let $f(\mathbf z)=z_1^{a}+\dots+z_{n-1}^{a}+z_{n-1}z_n^{a-1}$ with $a>1$. Then $\mathbb C^I,\,I=\{n\}$ is a vanishing coordinate subspace for $f$.  %$f$ is non-degenerate and locally tame.
Let $g_1(\mathbf z)=c_1z_1^{a-1}z_2+c_2z_2^a+\dots+c_nz_{n}^{a}$ and 
$g_2=c_1z_1^{2a}+\cdots+c_{n-1}z_{n-1}^{2a}+c_n z_{n-1}^2 z_n^{2a}$.
For generic coefficients $c_1,\dots,c_n$, $f, g_1, g_2$ are  locally tame non-degenerate functions and $\mathcal V_f=\mathcal V_{g_2}=\{\mathbb C^I\}$ and $\mathcal V_{g_1}=\{\mathbb C^J\}$ where $I=\{n\}$ and  $J=\{1\}$.
$\{f,g_1\}$  is a disjoint locally tame non-degenerate pair
% $f,g_1$ have  disjoint vanishing coordinate subspaces. % if $n\ge 4$.
while  $\{f,g_2\}$ is  a  locally tame non-degenerate pair. They have the common vanishing coordinate subspace $\mathbb C^{\{n\}}$.
%Note also that $f$ is not convenient but $f$ has an isolated singularity at the origin.
\section{Isolatedness of the critical value }

\subsubsection{Multiplicity condition}\label{multiplicity}
 We  slightly generalize the multiplicity condition which is introduced in \cite{fg-bar}. We say that
$H:=f\bar g$ satisfies {\em the  multiplicity condition}  if 
 there exists a good resolution $\pi: X\to \mathbb C^n$ of the holomorphic function $h:=fg$ such that
\begin{enumerate}
%\noindent
\item[(i)]
$\pi:X\setminus \pi\inv(V(h))\to \mathbb C^n\setminus V(h)$ is biholomorphic
and the divisor defined by $\pi^*(fg)=0$ has only normal crossing  singularities and the  respective strict transforms $\tilde V(f)$ of $V(f)$ and $\tilde V(g)$ of $V(g)$ are smooth.

%\noindent
\item[(ii) ]Put $\pi\inv(\mathbf 0)=\cup_{j=1}^s D_j$ where $D_1,\dots, D_s$ are  smooth compact divisors in $X$.
Denote  the respective multiplicities of $\pi^*f$  and $\pi^*g$ along $D_j$  by $m_j$ and $n_j$.
Then $m_j\ne n_j$ for $j=1,\dots,s$.
\end{enumerate}
Assume that there exists a regular simplicial cone subdivision $\Si^*$ of the dual Newton diagram
$\Ga^*(fg)$ and let $\hat\pi:X\to \mathbb C^n$ be the corresponding admissible toric modification. Let $\mathcal V^+$ be the set of strictly positive vertices  of $\Si^*$. Then it gives a good resolution of the function $fg$ and the compact exceptional divisors are bijectively correspond to 
$\{\hat E(P)\,|\ P\in \mathcal V^+\}$ (\cite{Okabook}). Recall that the multiplicity of $\hat\pi^*f$ and $\hat\pi^* g$ along the divisor
$\hat E(P)$ are given by $d(P,f)$ and $d(P,g)$ respectively.
We say that $\hat \pi$ satisfies {\em the toric multiplicity condition for $H$} if 
\[
d(P,f)\ne d(P,g),\quad \forall P\in \mathcal V^+.
\]
For further detail about the toric modification $\hat\pi:X\to \mathbb C^n$, we refer to \cite{Okabook}.
% (Lemma \ref{nearby fiber}.
\begin{Lemma}[ Isolatedness of the critical value, Lemma 3 \cite{fg-bar}]\label{smoothness}
Assume that $\{f,g\}$ is a locally tame and  non-degenerate complete intersection pair. Assume that there exists an admissible toric modification
$\hat \pi:X\to \mathbb C^n$ which satisfies the toric multiplicity condition.
%and  satisfies the  tame Newton multiplicity condition. % $(\sharp)$.
 Then 
 %it satisfies  the multiplicity condition and t
 there exist positive numbers $r_1$ such that $0$ is the unique  critical value of $H$ on
%the nearby fiber $V_\eta:=H\inv(\eta)$ has no mixed singularity in the ball
 $B_{r_1}^{2n}$.
\end{Lemma}
The proof follows by the exact same argument as Lemma 3,\cite{fg-bar}.
%%%%%
\subsubsection{A sufficient condition for the toric multiplicity condition}
We consider the following truncated cone. Let $h(\mathbf z)=\sum_{\nu}a_\nu \mathbf z^\nu$ be a holomorphic function which is not necessarily convenient.
Let $\Ga_+(h)$ be the convex hull of the union $\bigcup_{\nu,a_\nu\ne 0} \{\nu+(\mathbb R^+)^n\}$ as usual. The Newton boundary 
$\Ga(h)$ is defined by the union of compact faces of $\Ga_+(h)$. 
To give a sufficient condition for the multiplicity condition, 
we further consider following.
\begin{Definition} We define  the set $\Ga_{++}(h)$ and $\Int\Ga_{++}(h)$ as
\[
\Ga_{++}(h)=\{r\nu\,|\, r\ge 1,\,\nu\in \Ga(h)\},\, \Int\Ga_{++}(h)=\{r\nu\,|\, r>1,\,\nu\in \Ga(h)\}
\]
\end{Definition}
Note that $\Ga_{++}(h)\subset \Ga_+(h)$ and the equality holds if and only if $h$ is convenient.
The following gives a sufficient condition for the multiplicity condition.
\begin{Lemma} Assume $\{f,g\}$ is a locally tame non-degenerate complete intersection pair.
Suppose the following condition is satisfied.
\nl
$(\sharp)$: $\Ga(f)\subset \Int\,\Ga_{++}(g)$
or 
   $\Ga(g)\subset \Int\,\Ga_{++}(f)$. 
\nl
%This is a generalization of Newton multiplicity condition in \cite{fg-bar} for non-convenient functions $f,g$.
Then the multiplicity condition is satisfied with respect to any admissible toric modification. 
%Here $\Int\,\Ga_{++}(f)$ denote the  $\Ga_{++}(f)\setminus \Ga(f)$.
\end{Lemma}
See Figure \ref{Gamma++} which shows the situation $\Int\,\Ga_{++}(f)\supset \Ga(g)$.
The condition
$(\sharp)$ 
is a generalization of Newton multiplicity condition in \cite{fg-bar} for non-convenient $f$ and $g$.
We call $(\sharp)$ {\em the tame Newton multiplicity condition}.
%\begin{Proposition} Assume that $f,g$ satisfies the tame Newton multiplicity condition. Then any admissible toric modification $\hat \pi:X\to \mathbb C^n$ satisfies the toric multiplicity condition.
%\end{Proposition}
\begin{figure}[htb]
\setlength{\unitlength}{1bp}
\begin{picture}(600,300)(-100,0)
{\includegraphics[width=8cm, bb=0 0 595 842]{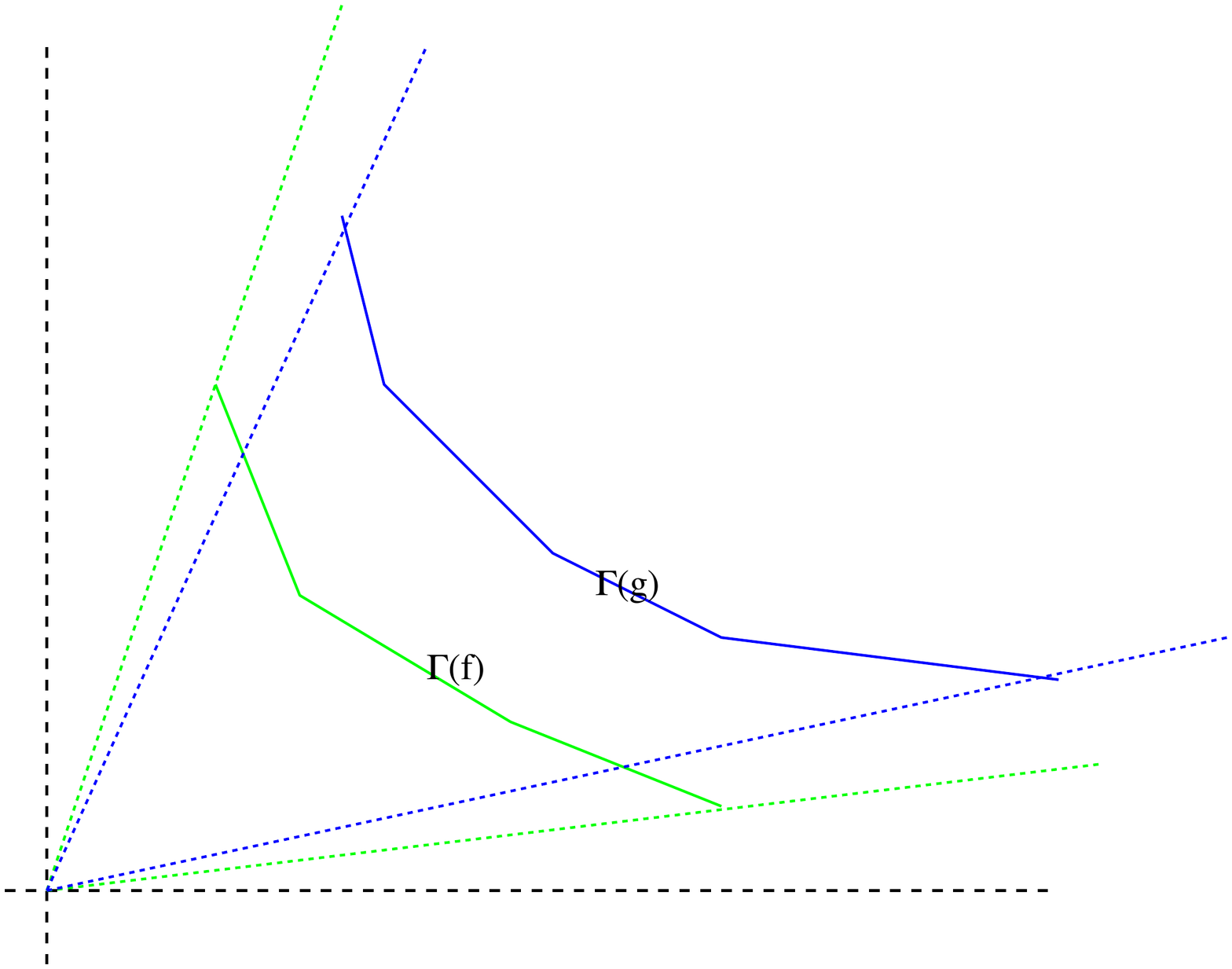}}
%\scalebox{0.3}{0.1}\includegraphics{3A6.eps}
\put(-140,-450){\tiny$\Ga_{++}(g)$}
\put(-255,-470){\tiny$\Ga_{++}(f)$}

%\put(10,70){$\delta$}
\end{picture}
\vspace{-2cm}
\caption{$\Ga(g)\subset \Int\,\Ga_{++}(f)$}\label{Gamma++}
\end{figure}
\begin{Example} \label{Ex5}
1. Assume that $f(\mathbf z)$ (respectively $g$) is a convenient function and assume that $\Ga(f)\cap \Ga(g)=\emptyset$ and  $\Ga(g)$ is above  $\Ga(f)$(resp. $\Ga(f)$ is above $\Ga(g)$). Then the tame Newton multiplicity condition is satisfied.

2.  Assume $\{f,g\}$ is a locally tame non-degenerate complete intersection pair and let $\hat \pi:X\to \mathbb C^n$ is an admissible toric modification. Let $\mathcal V^+$ be the strictly positive vertices of $\Si^*$. Consider the  mapping $\vphi_m:\mathbb C^n\to \mathbb C^n$
defined by $\vphi(\mathbf z)=(z_1^m,\dots, z_n^m)$ and put $f_m(\mathbf z):=\vphi^*f(\mathbf z)$
and $g_m(\mathbf z):=\vphi^*g(\mathbf z)$.  Then there exists a sufficiently large $m$ such that $\hat \pi:X\to \mathbb C^n$ satisfies the toric multiplicity condition for $f_m\bar g$ and $f\bar g_m$ respectively.
This follows from the canonical equality $d(P,f_m)=m\,d(P,f)$ and $d(P,g_m)=m\,d(P,g)$ and 
the stability of the dual Newton diagrams $\Ga^*(f)=\Ga^*(f_m),\,\Ga^*(g)=\Ga^*(g_m)$.

3. Let $f(\mathbf z)=c_1z_1^a+\dots+c_nz_n^a$.  Let $g(\mathbf z)$ be any locally tame non-degenerate function.  
Then $\{f,g\}$ is a locally tame non-degenerate pair for generic coefficients $c_1,\dots, c_n$ and satisfies the Newton multiplicity condition if
$a=1$. If $a>1$,  $\{f,g \prod_{i=1}^n z_i^a\}$ satisfies the Newton multiplicity condition, as $\Ga_{++}(f) \supset \Int\Ga_{++}(g)$.

%4. If $\{f,g\}$ satisfies the Newton multiplicity condition, $\{g,f\}$ also sati
\end{Example}

\begin{proof}[Alternative proof of Lemma \ref{smoothness}] Though Lemma \ref{smoothness} follows from Lemma 3, \cite{fg-bar} under the assumption of the tame Newton multiplicity condition ($\sharp$), it will be useful to give another proof which does not  use a resolution. We prove the assertion by   contradiction.  Assume that 
 the assertion does not hold. Then using  the Curve Selection
 Lemma (\cite{Milnor,Hamm1}),  we can find an analytic  path
$\bfz(t),\,0\le t\le 1$,
$\bfz(0)=\mathbf 0$ such that 
$H(\bfz(t),\bar\bfz(t))\ne 0$ for $t\ne 0$  and $\bfz(t)$ is a critical point
 of the function $H:\BC^n\to \BC$ for any $t$.
Using Proposition 1, \cite{OkaPolar} (see also Proposition 1, \cite{fg-bar}),
we can find an analytic  function $\la(t)$ whose values are  in $S^1\subset \BC$
such that 
\begin{eqnarray} %\label{Sing-cond1}
 \overline{\partial H}(\bfz(t),\bar\bfz(t))=\la(t)\, \bar \partial H(\bfz(t),\bar\bfz(t)).
\end{eqnarray}
Note that in our case we have
 \[\overline{\partial H}=\overline{\partial f}  g,\quad \bar{\partial }H=f\overline{\partial g}.
\]
Thus (\ref{Sing-cond1}) implies
\begin{eqnarray}\label{Sing-cond1}
g(\mathbf z(t))\overline{\partial f}(\mathbf z(t))=\la(t) f(\mathbf z(t))\overline{\partial g}(\mathbf z(t)).
\end{eqnarray}
Put $I=\{j\,|\,z_j(t)\not \equiv 0\}$. 
We may assume for simplicity that 
 $I=\{1,\dots,m\}$ and we consider the restriction  $H^I=H|\mathbb C^I$.  
%As $H(\bfz(t),\bar \bfz(t))= H^I(\bfz(t),\bar \bfz(t))$ $ \not  \equiv 0$,
By the assumption $\mathbf z(t)\notin V(H)$ for $t\ne 0$,  we
 see that $H^I\ne 0$. 
Consider the Taylor expansions of $\bfz(t)$ and  $\la(t)$:
\begin{eqnarray*}
 \bfz_i(t)&=&b_i \,t^{p_i}+\text{(higher terms)},\,b_i\ne 0,p_i>0,\quad
i=1,\dots, m\\
\la(t)&=&\la_0+\text{(higher terms)},\quad \la_0\in S^1\subset \BC.
\end{eqnarray*}
Consider the weight vector  $P=(p_1,\dots, p_m)$ and the point $\bfb=(b_1,\dots, b_m)\in \mathbb C^{*I}$
  and  also the face function $f^I_P$
of $f^I(\bfz,\bar\bfz)$.
Recall that $f_P^I$ and $g_P^I$  are defined by the partial sum of monomials in $f^I(\mathbf z,\bar{\mathbf z})$ where the monomials 
have the minimal degree $d(P;f^I)$ and $d(P;g^I)$ respectively.
% (See \cite{OkaMix})
% and put $d=d(A;f^I)$ and $\bfb=(b_1,\dots, b_m)$.
%For notation's simplicity, we put $f^I=f$.
% and $m=n$ .
Then 
%putting  $d=d(A;f^I)$, 
we have 
\begin{eqnarray*}
f(\mathbf z(t))&=&f_P^I(\mathbf b)t^{d(P;f^I)}+\text{(higher terms)},\\
g(\mathbf z(t))&=&g_P^I(\mathbf b)t^{d(P;g^I)}+\text{(higher terms)},\\
\frac{\partial f}{\partial z_j}(\bfz(t))&=&\frac{\partial f^I_P}{\partial
 z_j}(\bfb)\,t^{d(P;f)-p_j}+\text{(higher terms)},\,j\in I\\
 \frac{\partial g}{\partial  z_j}(\bfz(t))&=&\frac{\partial g^I_P}{\partial
 z_j}(\bfb)\,t^{d(P;g)-p_j}+\text{(higher terms)},\,j\in I.
 \end{eqnarray*}
%where  $d=d(A;f^I)$.
The equality (\ref{Sing-cond1}) says that 
\begin{eqnarray}\label{eq2}
 {\overline{\frac{\partial f^I}{\partial z_j}}
 \,(\bfz(t),\bar\bfz(t))}=\la(t)\,\frac{f(\mathbf z(t))}{g(\mathbf z(t))}
\overline{ \frac{\partial g^I}{\partial  z_j}}(\bfz(t),\bar\bfz(t) ),\,\,j=1,\dots, m.
\end{eqnarray}
%Note that 
%\[\ord\,f(\mathbf z(t))\ge d(A;f),\quad \ord\, g(\mathbf z(t))\ge d(A;g).
%\]
Put $\ell:=\ord\,f(\mathbf z(t))-\ord\,g(\mathbf z(t))=\ord_t\frac{f(\mathbf z(t))}{g(\mathbf z(t))}$.
Here $\ord\, \vphi(t)$ of a Laurent series $\vphi(t)$ is by definition the lowest degree of the series $\vphi(t)$. Thus 
$\lim_{t\to 0}\vphi(t)/t^{\ord\,\vphi(t)}$ is a non-zero number.
Note that 
\[\begin{split}&\ord\,f(\mathbf z(t))\ge d(P;f),\quad \ord\, g(\mathbf z(t))\ge d(P;g),\\
&\ord\,\frac{\partial f}{\partial z_j}(\bfz(t))\ge d(P;f)-p_j,\,\ord\, \frac{\partial g}{\partial z_j}(\bfz(t))\ge d(P;g)-p_j.
\end{split}
\]
\indent%\newline
Case 1.  (a) $f_P^I(\mathbf b)=0,\,g_P^I(\mathbf b)\ne 0$ or (b) $f_P^I(\mathbf b)\ne 0,\,g_P^I(\mathbf b)=0$.
In the case of (a), $\ell>d(P;f^I)-d(P;g^I)$ and (\ref{eq2}) says that $\partial f_P^I(\mathbf b)=0$. 
This implies $\mathbf b$ is a critical point of $f_P^I$ and a contradiction to the non-degeneracy assumption.
In case of (b), $\ell<d(P;f^I)-d(P;g^I)$ we get   $\partial g_P^I(\mathbf b)=0$ and we get also a contradiction to the non-degeneracy assumption of $g$.

%\noindent
Case 2. $f_P(\mathbf b)=0,\,g_P(\mathbf b)= 0$. Then $\mathbf b\in V(f_P,g_P)$.
If $\ell\ne d(P;f)-d(P;g)$, we get a  contradiction $\partial f_P^I(\mathbf b)=0$
or $\partial g_P^I(\mathbf b)=0$. So we  assume that $\ell= d(P;f)-d(P;g)$.
Then we see that $\partial f_P^I(\mathbf b), \partial g_P^I(\mathbf b)$ are linearly dependent over $\mathbb C$
but this is a contradiction 
to the non-degeneracy assumption of the intersection variety $V(f,g)$.

Case 3. Assume that $f_P(\mathbf b)\ne 0,\,g_P(\mathbf b)\ne 0$.  Then we have $\ell=d(P;f)-d(P;g)$ and 
\[g_P(\mathbf b)\overline{ \frac{\partial f_P}{\partial  z_j}}(\mathbf b)=\la_0 
f_P(\mathbf b)\overline{ \frac{\partial g_P}{\partial  z_j}}(\mathbf b),\,\,j=1,\dots, m.
\]
Multiplying $p_j \bar b_j$ to the both side and summing up for $j=1,\dots, m$,  we use the Euler equalities of $f_P$ and $g_P$,
\[
d(P,f^I)f_P^I (\mathbf b)=\sum_{j=1}^m p_j b_j\frac{\partial f_P^I}{\partial z_j}(\mathbf b),\quad
d(P,g^I)g_P^I (\mathbf b)=\sum_{j=1}^m p_j b_j\frac{\partial g_P^I}{\partial z_j}(\mathbf b),
\]
to get the equality
\[\begin{split}
&d(P;f)\overline{f_P}(\mathbf b) g_P(\mathbf b)=\la_0 d(P;g)f_P(\mathbf b) \overline{g_P}(\mathbf b)\,\,\text{or}
\end{split}\]
This gives an absurd equality:
\[
1\ne\frac{d(P;f)}{d(P;g)}=
%\la_0 \frac{f_P(\mathbf b)\overline{ g_P}(\mathbf b)}
%{\overline{f_P}(\mathbf b)g_P(\mathbf b)}=
\left|\la_0 \frac{f_P(\mathbf b)\overline{ g_P}(\mathbf b)}
{\overline{f_P}(\mathbf b)g_P(\mathbf b)}\right|=
1.
%\end{split}
\]
The first inequality follows from the tame Newton multiplicity condition. The last equality is
due to $|\la_0|=1$.
\end{proof}

\begin{Remark}
If $f$ is locally tame and  non-degenerate and if $\mathbb C^I$ is not a vanishing coordinate subspace for $f$, $f^I$ is also locally tame and  non-degenerate as a function on $\mathbb C^I$. %Recall that the restriction $f|_{\mathbb C^I}$ is denoted as $f^I$.
See the argument in Proposition (1.5), Chapter III \cite{Okabook}. Locally tameness has been defined for mixed functions (Definition 2.7, \cite{EO}). If a holomorphic function $f(\mathbf z)$ is locally tame, it is also locally tame as a mixed function.
\begin{comment}
 If $f=g=0$ is locally tame and  non-degenerate and assume that $\mathbb C^{I}\in \mathcal V_f$ 
and  $\mathbb C^{I}\notin \mathcal V_g$  with $I=I(P)$,  then  $g_P=g^I$ and it contains only $z_i,\,i\in I$. Thus 
$f_P=g_P=0$ is also a smooth complete intersection variety in 
$\mathbb C^{*n}\cap B_{r_I}^{2n}$.
\end{comment}
% as long as $\sum_{i\in I} |z_i|^2$ is small enough.

\end{Remark}
 %%%%%%%%%%%%
%%%%%%%%%%%%%%%%%%%%%%

\section{ Fibration problem for function $f \bar g$}
%This paper is a continuation of  \cite{fg-bar}.
We study the existence problem for the Milnor fibration of the mixed function $H(\mathbf z,\bar{\mathbf z}):=f(\mathbf z)\bar g(\mathbf z)$
in a more general situation. In this paper,  we do not assume the convenience of $f$ and $g$  and therefore $V(f)$ or $V(g)$ may have non-isolated singularities at the origin. There are also interesting  works from more general viewpoint in 
Parameswaran and Tibar \cite{Param-Tibar,Param-Tibar-revised} and Araujo dos Santos, Ribeiro and Tibar \cite{ART1} where  authors consider the case of critical values being not isolated.
\subsection{Canonical stratification }
We assume that $\{f,g\}$ is a locally tame non-degenerate complete intersection pair.
Consider the hypersurface $V(fg)=V(f)\cup V(g)$. 
Note that the mixed hypersurface $V(f\bar g)$ is equal to $V(fg)$ as  real algebraic varieties. We consider the following canonical stratification $\mathcal S$ of  $\mathbb C^{*n}$ which also give a stratification of $V(fg)$.
Put $V^{*I}(f)=V(f^I)\cap\mathbb C^{*I}$ if $\mathbb C^I$ is not a vanishing coordinate subspace.
% and $\mathcal S^I(f)=\mathbb C^{*I}$ if $\mathbb C^I\in \mathcal V_f$. 
\newline
Here $\mathbb C^{*I}=\{(z_i)\in \mathbb C^I\,|\, \forall z_i\ne 0,\,i\in I\}$.
We first  define a stratification $\mathcal S^I$ of $\mathbb C^{*I}$ as  follows.
\[\begin{split}
%&\mathcal  S^I=\\
&\left\{\mathbb C^{*I}\setminus  ( V^{*I}(f)\cup V^{*I}(g)),   V^{*I}(f)', V^{*I}(g)', \ V^{*I}(f)\cap   V^{*I}(g)\right\}, \text{if}\,
 f^I\ne 0,\, g^I\ne 0\\
&\left\{\mathbb C^{*I}\setminus V^{*I}(f),  V^{*I}(f)\right\}, \text{if}\,g^I\equiv 0,f^I\ne 0\\
&\left\{\mathbb C^{*I}\setminus V^{*I}(g),  V^{*I}(g)\right\}, \text{if}\, f^I\equiv 0,g^I\ne 0\\
&\left\{\mathbb C^{*I}\right\},\quad  \text{if}\,f^I\equiv 0,g^I\equiv 0.
\end{split}
%\mathcal S=\cup_{I}\mathcal S^I.
\]
and  we define $\mathcal S=\cup_{I}\mathcal S^I$. If $\{f,g\}$ is a disjoint locally tame non-degenerate pair, the last case  does not exist.
Here $V^{*I}(f)'= V^{*I}(f)\setminus V^{*I}(g)$ and $V^{*I}(g)'=V^{*I}(g)\setminus V^{*I}(f)$.
$V^I(f)$ is empty only if $f^I$ is a monomial.
%n the case $\mathcal S^I(f)=\mathbb C^{*I}$ and $\mathbb C^I\notin \mathcal V_g$ for example,  we understand as 
%$\mathcal S^I=\{\mathcal S^I(g),\mathbb C^{*I}\setminus \mathcal S^I(g)\}$.

We call $\mathcal S$  {\em the canonical toric strafitication }of $V(fg)=V(f\bar g)$. Note that $\mathcal S$ is a complex analytic stratification.
%There is another important  condition for $H=f\bar g$ to be fibered.
%%%%%%%%%%%%%%%%%
%\begin{proof} The proof is exactly the same as that of Lemma 3, \cite{fg-bar}.
%\end{proof}
\subsection{Transversality and Thom's $a_f$-regularity}
% and Hamm-L\^e type Lemma}
We use the notation $V(H, \mathbf z):=H\inv(H(\mathbf z))$ hereafter.
Another key condition for the existence of the Milnor fibration is the transversality of the nearby fibers $H\inv(\eta),\,\eta\ne 0$ and the sphere
$S_r^{2n-1}$.  Assume that $0$ is the unique critical value of $H$  in $B_{r_1}^{2n}$. 
\vspace{.2cm}\newline\noindent
{\bf Transversality of nearby fibers}:
For any pair $r_2\le r_1$, there exists a positive number $\de$ such that 
for any $r,\, r_2\le r\le r_1$ and non-zero $\eta$ with  $|\eta|\le \de$, $H\inv(\eta)$ and  $S_r^{2n-1}$ intersect transversely.
%\vspace{.2cm} 
This condition follows if  $H$ satisfies the Thom's $a_f$-regularity (See for example, Proposition 11, \cite{OkaAf}).
Recall that $H$ satisfies $a_f$-condition at the origin if there exists a stratification  $\mathcal S$ of $H\inv(0)\cap B_{r_1}^{2n}$ for some $r_1>0$ such that for any sequence $\mathbf q_\nu,\,\nu=1,2,\dots$
which converges $\mathbf q_0\in M,\,M\in \mathcal S$ and $\mathbf q_0\ne \mathbf 0$, the limit of the tangent space $T_{\mathbf q_\nu}V(H,\mathbf q_\nu)$ (if it exists) includes the tangent space of $M$ at $\mathbf q_0$.

% We assume further that any sphere of radius $r\le r_1$ intersects transversely with $V(f), V(g)$ and $V(f,g)
%The following follows from Lemma 3 and Corollary 4.1, \cite{Param-Tibar}.
\begin{Theorem}\label{af-regularity}
Assume that  either (i) $\{f,g\}$ is a  locally tame non-degenerate complete intersection pair which satisfies also  
 the tame Newton multiplicity condition $(\sharp)$  %\nl% or  the tame Newton multiplicity condition $(\sharp)$
 or (ii) $\{f,g\}$ is a disjoint locally tame non-degenerate complete intersection pair.
 In the case (ii), we assume also that $H$ has a unique critical value  $0$  in  a ball $B_{r_1}^{2n}$. 
 Then $H=f\bar g$ satisfies $a_f$-regularity.
\end{Theorem}
%\begin{proof} %{Proof of Lemma \ref{af-property}}\label{proofofaf}
Note that in case (i), the tame  Newton multiplicity condition guarantees the isolatedness of the critical value of $H$.
For the proof, we consider the canonical toric stratification $\mathcal S$ on 
$V(fg)$.  We choose  $r_0, r_1\ge r_0>0$ sufficiently small so that 
for any $r\le r_0$, the canonical toric strata are smooth in $B_r^{2n}$ and 
any sphere $S_\rho^{2n-1}$ with $0<\rho\le r_0$ meets  transversally with every strata of $\mathcal S$ of positive dimension.
%We  prove the assertion by contradiction.
We use  Curve selection lemma (see \cite{Milnor, Hamm1}).
Suppose we have a real analytic curve
$\mathbf z(t),\,0\le t\le 1$   such that  $\mathbf z(0)=\mathbf a\in V(H)\cap B_{r_0}^{2n}$, $\mathbf a\ne \mathbf 0$ and $\mathbf z(t)\in \mathbb C^n\setminus V(H)$ for $t>0$.  Put $K=\{i\,|\, z_i(t)\not\equiv 0\}$ and write the expansion as 
\[\begin{split}
z_i(t)&=\al_it^{p_i}+\text{(higher terms)}, \, \al_i\ne 0,\, i\in K\\
&\equiv 0,\,\, i\notin K
\end{split}
\]
%such that $\mathbf z(t)\notin V(H)$ for $t\ne 0$ and 
%$\mathbf a\in M\subset \mathbb C^{*I}\cap V(H)$ where
Let  $M\in \mathcal S^I$ be the stratum which contains $\mathbf a$.
We have to show that  the limit of the tangent space of the fiber $V(H,\mathbf z(t))$  at $\mathbf z(t)$  for $t\to 0$ contains the tangent space  of the stratum $M$ 
at $\mathbf a$.  
%Here $M$ is the stratum which contains $\mathbf a$. 
%We assume that $\mathbf a\in V(H)$.
The restriction of $f,g$ and $H$ on $\mathbb C^K$ satisfy also the locally tame non-degenerate assumption.
As the argument    for the proof is exactly the same,
 we assume for simplicity  that $K=\{1,\dots, n\}$ hereafter. That is, we assume that  $\mathbf z(t)\in \mathbb C^{*n}$ for $t\ne 0$ and $\mathbf z(0)=\mathbf a$.
Put  $ P=(p_1,\dots, p_n)$ and 
\begin{eqnarray}\begin{cases}
&I:=\{i\,|\, p_i=0\},\,\,\\
& J=I^c=\{1,\dots,n\}\setminus I,\quad \mathbf w:=(\al_1,\dots,\al_n)\in \mathbb C^{*n}. %\label{alpha-vector}
\end{cases}
\end{eqnarray}
Note that $p_i=0$ if and only if $i\in I$. %Put $\mathbf z(0)=\mathbf a$.
Thus $ \mathbf a=\mathbf w_I$ and $0\ne \|\mathbf a\|\le r_0$.
%Let $M$ be the stratum which contains $\mathbf a$.
We will show that 
\[
\lim_{t\to 0}T_{\mathbf z(t)}V(H,\mathbf z(t))\supset T_{\mathbf a}M.
\]
We use the key property that the tangent space of the level hypersurface  $V(H,\mathbf z(t))$ at $\mathbf z(t)$ 
%of the mixed function $H=f\bar g$ 
contains 
%real codimension 2 subspace,
 the intersection of two tangent spaces of the level  complex hypersurfaces $V(f,\mathbf z(t))$ and $V(g,\mathbf z(t))$
by Proposition 14, \cite{fg-bar}. Here we are assuming that $r_0$ is sufficiently small so that $0$ is the only critical value for $f$ and $g$ on $B_{r_0}^{2n}$.
We divide the situation into three cases.
\begin{enumerate}
\item[(a)] $\mathbb C^I\notin \mathcal V_f\cup \mathcal V_g$, i.e. $f^I\not \equiv ,g^I \not \equiv 0$.%is not a vanishing coordinate subspace for both $f$ and $g$.
\item[(b)] $\mathbb C^I\in \mathcal V_f$  and $\mathbb C^I\notin \mathcal V_g$,  i.e. $f^I\equiv 0,\,g^I\not \equiv 0$.
\item[$(b)'$] $\mathbb C^I\in \mathcal V_g$  and $\mathbb C^I\notin \mathcal V_f$, i.e. $f^I\not\equiv 0$, $g^I\equiv 0$.
%\item[(c)] $\mathbb C^I\in \mathcal V_f\cap\mathcal V_g$, i.e. $f^I \equiv  0, \,g^I\equiv 0$.
\item[(c)] $\mathbb C^I\in \mathcal V_f\cap \mathcal V_g$, i.e. $f^I\equiv 0$, $g^I\equiv 0$.

\end{enumerate}
As (b) and  $(b)'$ is symmetric, it is enough to  consider three cases (a), (b) and (c).  % and (c).

We first consider the case (a).
The case (a) can be divided into two subcases:
\begin{enumerate}
\item[(a-1)]
$\mathbf a\in M= V^{*I}(f)\cap V^{*I}(g)$.
\item[ (a-2)] $\mathbf a\in V^{*I}(f)'=V^{*I}(f)\setminus V^{*I}(g)$, or
$\mathbf a\in V^{*I}(g)'=V^{*I}(g)\setminus V^{*I}(f)$.
%\item[(a-3)] $\mathbf a\in \mathcal S^I(g)'$.
\end{enumerate}
In the case (a-1),  $\mathbf a$ is a non-singular point of  $\mathbf a\in V(f,g)$.
As the tangent space $T_{\mathbf z(t)} V(f,\mathbf z(t))$ converges to $ T_{\mathbf a}V(f)$ which includes $T_{\mathbf a}V^{*I}(f)$ and 
$T_{\mathbf z(t)}V(g,\mathbf z(t)))$ converges to $ T_{\mathbf a}V(g)$ which includes  $T_{\mathbf a}V^{*I}(g)$ and
$T_{\mathbf a}M=T_{\mathbf a}V^{*I}(f)\cap T_{\mathbf a}V^{*I}(g)$ by the Newton non-degeneracy assumption,
the assertion follows from Proposition 14, \cite{fg-bar}.

 In the case (a-2), $\mathbf a\in V^{*I}(f)'$  or $\mathbf a\in V^{*I}(g)'$,
 $\mathbf a$ is  a non-singular point of $V(H)$
 and the assertion is obvious from the continuity of the tangent space.

Consider the case (b). Thus we assume that $\mathbb C^I\in \mathcal V_f\setminus\mathcal V_g$.
By the local tameness assumption, the limit of the normalized  holomorphic  gradient vector 
$\lim_{t\to 0} \overline{\partial\, f}(\mathbf z(t))/\|  \overline{\partial\, f}(\mathbf z(t))\|$ along $\mathbf z(t)$ is a vector 
in $\mathbb C^{J}$. Here $J=\{1,\dots,n\}\setminus I$.  (Recall $\partial f(\mathbf z)=(\frac{\partial f(\mathbf z)}{\partial z_1},\dots, \frac{\partial f(\mathbf z)}{\partial z_n})$.)
Thus the limit of the tangent space of $V(f,\mathbf z(t))$ contains $\mathbb C^I$ by the local tameness assumption.
There are two subcases.
\begin{enumerate}
\item[(b-1)] $\mathbf a\in V^{*I}(g)$, or
\item[(b-2)] $\mathbf a\in \mathbb C^{*I}\setminus V^{*I}(g)$.
\end{enumerate}%
 \noindent
Note that $M=V^{*I}(g)$ in the case (b-1)  and  $M=\mathbb C^{*I}\setminus V^{*I}(g)$  in the case (b-2) respectively.
In the case of (b-1),  the limit of the normalized vector of   $\overline{\partial f}(\mathbf z(t))$ is a vector in $\mathbb C^{J}$ by the local tameness assumption of $f$. Thus 
the limit of $T_{\mathbf z(t)}V(f,\mathbf z(t))$  includes $\mathbb C^I$.
On the other hand,   as $\overline{\partial g^I}(\mathbf a)$ is non-zero, $T_{\mathbf a}V(g)$ is transverse  to $\mathbb C^I$ at $\mathbf a$.
Thus for any sufficiently small $t$, they are transverse and the limit of the intersection of  two tangent space of the tangent space of $V(f,\mathbf z(t))$ and $V(g,\mathbf z(t))$ contains $T_{\mathbf a}V^{*I}(g)$.

Now
we consider the case (b-2).  We claim that the limit of the tangent space $T_{\mathbf z(t)}V(H,\mathbf z(t))$ 
%of the level hypersurface at $\mathbf z(t)$ by $H$ 
 includes $\mathbb C^I$, the tangent space of the stratum $M=\mathbb C^{*I}\setminus V^{*I}(g)$ at $\mathbf a$.
%Write $f=f_1+if_2, g=g_1+ig_2$ using real valued functions.
%Then $f\bar g=k_1+ik_2=f_1g_1+f_2g_2+i(f_2g_1-f_1g_2)$.
%The normalized gradient vector of $k_1$ at $\mathbf z(t)$ is 
First we prepare a sublemma.
\begin{Sublemma}\label{sublemma}
 Let  $f$ be a holomorphic function and write  $f(z)=k(\mathbf z,\bar{\mathbf z})+i \ell(\mathbf z,\bar{\mathbf z})$ where $k=\Re\, f,\,\ell=\Im\, f$.
Then we have $\bar \partial k=\frac 12 \overline{\partial f}$
and $\bar \partial \ell=\frac i{2 }\overline{\partial f}$. In particular, two gradient vectors
$\bar \partial k$ and $\bar \partial \ell$ are linearly dependent over $\mathbb C$ but  linearly independent over $\mathbb R$
at a non-critical point $\mathbf z$ of $f$.
\end{Sublemma}
%\begin{proof}
The assertion follows from the identities:
\[
\overline{\partial k}=\bar\partial k,\, \overline{\partial \ell}=\bar \partial \ell,\,\, \bar \partial f=\bar\partial k+i\bar\partial \ell=0,
\,\, \partial f=\partial k+i\partial \ell.
\]
%\end{proof}
Put $p_{min}=\min\{p_j\,|\, j\notin I\}$.
First we can write
\begin{Lemma} 
%Let $s$ be the order of $\overline{\partial f}(\mathbf z(t))$. 
The orders of  $\bar \partial {\Re f}(\mathbf z(t))$ and $\bar \partial {\Im f}(\mathbf z(t))$ are equal  to the order of $\overline{\partial f}(\mathbf z(t))$. 
Put $s=\order\,\overline{\partial f}(\mathbf z(t))$. Then $s$ and strictly less than $d(P;f)-p_{min}$. We can write further as follows.
\[
\begin{split}
\overline{\partial f}(\mathbf z(t))&=\mathbf v t^s+\text{(higher terms)},\, \, \exists\mathbf v\in \mathbb C^{J}\\
\bar\partial {\Re f}(\mathbf z(t))&=\frac 12 \mathbf v t^s+\text{(higher terms)}\\
\bar\partial {\Im f}(\mathbf z(t))&=\frac i2 \mathbf v t^s+\text{(higher terms)}
\end{split}
\]
In particular,
$lim_{t\to 0} T_{\mathbf z(t)} V(f,\mathbf z(t))$ is the complex orthogonal of $\mathbf v$.
\end{Lemma}
Now we are ready to analyze the  case (b-2).  Note that 
the limit of normalized gradient vector 
$ \overline{\partial f}(\mathbf z(t))$ is $\mathbf v/\|\mathbf v\|$.
For  a vector $\mathbf v$, let $\mathbf v^{\perp_{\mathbb C}}$ be the subspace of $\mathbb C^n$ which are complex orthogonal to $\mathbf v$.
Namely
${\mathbf v}^{\perp_\mathbb C}=\{\mathbf w\in \mathbb C^n\,|\, (\mathbf w,\mathbf v)=0\}$. Now we claim 
%
%For brevity 
%Put $b=\bar g(\mathbf a)$. % and $\hat f(\mathbf z):=f(\mathbf z)\bar b$ and $\hat g(\mathbf z)=g(\mathbf z)/b$.
%Then $H=\hat f\bar{\hat g}$. Note that $\hat f$ is also locally tame and non-degenerate. The assertion follows from $a_f$-regularity of $f$.
\begin{Assertion}
Assume $\mathbf a\in \mathbb C^{*I}\setminus V^{*I}(g)$. Then 
$\lim_{t\to 0}T_{\mathbf z(t)}V(H,\mathbf z(t))
% ( {\mathbf v} b)^{\perp_{\mathbb C}}
=\lim_{t\to 0}T_{\mathbf z(t)} V(f,\mathbf z(t))$.
\end{Assertion}
\begin{proof}
 Put $b:=\bar g(\mathbf a)$ and write $b=b_1+ib_2$ with $b_1,b_2\in \mathbb R$. 
 First we use the equalities:
 \[\Re(H)=\Re(f)\Re(\bar{ g})-\Im(f)\Im(\bar{ g}), \quad
  \Im(H)=\Re( f)\Im(\bar{ g})+\Im(f)\Re(\bar{g}).
\]
Then the gradient vectors are given as 
\[
\begin{split}
%\Re(H)&=\Re(f)\Re(\bar{ g})-\Im(f)\Im(\bar{ g}), \\
%  \Im(H)&=\Re( f)\Im(\bar{ g})+\Im(f)\Re(\bar{g}),\\
\bar\partial\Re(H)(\mathbf z(t))
 &= (\bar\partial\Re( f)\Re({\bar{g}})(\mathbf z(t))+(\Re( {f})\bar\partial\Re(\bar{ g}))(\mathbf z(t))\\
 &-(\bar\partial\Im(f)\Im(\bar{g}))(\mathbf z(t))- (\Im( f)\bar\partial\Im(\bar{ g}))(\mathbf z(t))\\
&\equiv b_1\bar\partial \Re( f)(\mathbf z(t)) -  b_2\bar\partial\Im f(\mathbf z(t))\,\,\modulo \, (t^{s+1})\\
%&=(\mathbf v b_1+i\mathbf v b_2) t^s+\text{(higher terms)}\\
&\equiv \frac{\mathbf v\bar b}2 t^s\, \,\modulo \, (t^{s+1})\\
\bar\partial\Im( H)(\mathbf z(t))&= (\bar\partial\Re(f )\Im{\bar{ g}})(\mathbf z(t))+(\Re(  f)\bar\partial\Im(\bar{g}))(\mathbf z(t))\\
 & +\bar\partial\Im(f)\Re(\bar{ g})(\mathbf z(t)) + \Im(f)\bar\partial\Re(\bar{ g})(\mathbf z(t))\\
%\left(\bar\partial\Im( \hat f )\bar{\hat g}+\bar {\hat f}\bar\partial \Im(\hat g)\right)\\
&\equiv b_2(\bar\partial\Re(f )(\mathbf z(t))+
b_1\bar\partial \Im(f)(\mathbf z(t))\, \,\modulo\,(t^{s+1})\\
%&=i\mathbf v \bar b t^s+\text{(higher terms)}\\
&\equiv \frac{(i\mathbf v  \bar b)}2 t^s\, \,\modulo\,(t^{s+1})\\
\end{split}
\]
%\end{comment}
%as $\hat g(\mathbf a)=1$
and therefore the normalized vector of these  gradient vectors $\bar\partial\Re(H)(\mathbf z(t))$ and 
$\bar\partial\Im(H)(\mathbf z(t))$
 converges to the vectors

\[
\frac{\mathbf v \bar b}{\| \mathbf v \bar b\|},\,\quad\,i\frac{\mathbf v\bar b}{\|\mathbf v\bar b \|}
\]
respectively. This implies
the limit of  the tangent space $T_{\mathbf z(t)} V (H,\mathbf z(t))$ is  the real orthogonal of the real 2-dimensional subspace span by these two vectors, that is nothing but the complex subspace $\mathbf v^{\perp_{\mathbb C}}$
which  is equal to  the limit of
$T_{\mathbf z(t)}V(f,\mathbf z(t))$. The proof of the assertion for (b-2) is now completed.  The case $\{f,g\}$ is a disjoint tame non-degenerate
complete intersection pair is now proved.

Now we consider the last case (c) $\mathbb C^I\in \mathcal V_f\cap \mathcal V_g$. 
%his happens if $f,g$ has a common vanishing coordinate subspace.
Recall that $\mathbf w=(\al_1,\dots,\al_n)$.
We divide the situation into three subcases.
\begin{enumerate}
\item[(c-1)]$f_P(\mathbf w)=g_P(\mathbf w)=0$.
\item[(c-2)] $f_P(\mathbf w)=0$ and $g_P(\mathbf w)\ne 0$ or  (c-2)' $f_P(\mathbf w)\ne 0$ and $g_P(\mathbf w)= 0$.
\item [(c-3)]$f_P(\mathbf w)\ne 0,\,g_P(\mathbf w)\ne 0$.
\end{enumerate}

We restate the assertion as
 the following lemma. 
%\end{proof}
\begin{Lemma}
Assume that $\mathbb C^{I}\in \mathcal V_f\cap \mathcal V_g$. The the limit of the tangent space $T_{\mathbf z(t) }V(H,\mathbf z(t))$
includes $\mathbb C^I$ as a subspace.
\end{Lemma}
\begin{proof}
First assume that $f_P(\mathbf w)=g_P(\mathbf w)=0$.
Put $\overline{\partial f}(\mathbf z(t))=(u_1(t),\dots, u_n(t)) $ and 
$\overline{\partial g}(\mathbf z(t))=(v_1(t),\dots, v_n(t))$. 
We can write as
\[\begin{split}
u_j(t)&=\overline{\frac{\partial f_P}{\partial z_j}}(\mathbf w)t^{d(P;f)-p_j}+\text{(higher terms)}\\
v_j(t)&=\overline{\frac{\partial g_P}{\partial z_j}}(\mathbf w)t^{d(P;g)-p_j}+\text{(higter terms)}.
\end{split}
\]
Put $o_f$ and $o_g$ be the orders of $\overline{\partial f}(\mathbf z(t))$ and $\overline{\partial g}(\mathbf z(t))$ respectively. That is 
$o_f=\min\,\{\ord_t u_i(t)\,|\, i=1,\dots, n\}$ and 
$o_g=\min\{\ord\, v_i(t)\,|\, i=1,\dots,n\}$. 
Then the limit of $\overline{\partial f}(\mathbf z(t))$  and $\overline{\partial g}(\mathbf z(t))$ up to   scalar multiplications
%$\mathbb P^{n-1}$
 are represented  respectively by
\begin{eqnarray}\label{limit}
\lim_{t\to 0}\frac 1{t^{o_f}}\overline{\partial f}(\mathbf z(t)),\quad \lim_{t\to 0}\frac 1{t^{o_g}}\overline{\partial g}(\mathbf z(t)).
\end{eqnarray}
We denote these limit vectors as 
$\lim_{t\to 0}^{(n)}\overline{\partial f}(\mathbf z(t))$  and $\lim_{t\to 0}^{(n)}\overline{\partial g}(\mathbf z(t))$.
If these two limits are linearly independent over $\mathbb C$, the intersection 
\[
T_{\mathbf z(t)}V(f,\mathbf z(t))\cap T_{\mathbf z(t)}V(g,\mathbf z(t))
\]
 converges to the the complex orthogonal subspace to these two limit vectors.
 That is,
 \[
 <\overline{\partial f}(\mathbf z(t)),\overline{\partial g}(\mathbf z(t))>^{\perp_{\mathbb C}}\mapsto 
<{\lim_{t\to 0}}^{(n)}\overline{\partial f}(\mathbf z(t)),{\lim_{t\to 0}}^{(n)}\overline{\partial g}(\mathbf z(t))>^{\perp_{\mathbb C}}.
 \]
The problem happens if these two limits are linearly dependent.
We use a similar argument as the one which is used  in the proof of  Theorem 20, \cite{OkaAf} or 
Theorem 3.14, \cite{EO}
to solve this problem.
For the simplicity of the argument, we assume that $J=\{1,\dots,m\}$ and $I=\{m+1,\dots, n\}$
and we assume that 
\[p_1\ge p_2\ge \cdots\ge p_m>0,\, p_{m+1}=\cdots=p_n=0.
\]
%Put $p_m:=\min\,\{p_i\,|\, i\le m\}$.
Note that $p_{min}=p_m$ under the above assumption and
\[
\begin{split}
&\ord\,u_j(t)\ge d(P;f)-p_j,\quad \ord\, v_j(t)\ge d(P;g)-p_j,\,j=1,\dots,m
\end{split}
\]
while for $j\ge m+1$,
\[\begin{split}
&\ord\,u_jt)\ge d(P;f),\quad \ord\, v_j(t)\ge d(P;g),\,j=m+1,\dots,n.
\end{split}
\]
Now we consider first (c-1): $f_P(\mathbf w)=g_P(\mathbf w)=0$.
By the locally tame non-degeneracy assumption, there exists $1\le a, b\le m,\,a\ne b$ so that  we have
\begin{eqnarray}\label{non-degenerate}
&&\det
\left(
\begin{matrix}
\overline{\frac{\partial f_P}{\partial z_a}}(\mathbf w)&\overline{\frac{\partial f_P}{\partial z_b}}(\mathbf w)\\
\overline{\frac{\partial g_P}{\partial z_a}}(\mathbf w)&\overline{\frac{\partial g_P}{\partial z_b}}(\mathbf w)\\
\end{matrix}
\right)\ne 0.
%&&\text{therefore}\,\,o_f\le d(P;f)-p_m,\, o_g\le d(P;g)-p_m \notag\\
%quad \exists a, \exists b\le m,\, a\ne b.
\end{eqnarray}
Here we  assume that $a\ne b$ but we do not assume that $a<b$.
In particular, 
\[\begin{split}
o_f\le \max\{d(P;f)-p_a,d(P;f)-p_b\}\le d(P;f)-p_m,\, \\
o_g\le \max\{d(P;g)-p_a,d(P;g)-p_b\}\le d(P;g)-p_m.
\end{split}
\]
%Here $\mathbf w=(\al_1,\dots, \al_n)$, as in (\ref{alpha-vector}).
For simplicity, we may assume that $o_f\le o_g$ and consider 
\[
\ell_0:=\min\{j\,|\, \ord\, u_j(t)=o_f\},\quad m_0:=\min\{j\,|\, \ord\,v_j(t)=o_g\}.
\]
We call $\ell_0$, $m_0$  {\em the leading indices of $\overline{\partial f}(\mathbf z(t)$ and $\overline{\partial g}(\mathbf z(t))$.}

Case 1. Assume that $\ell_0\ne m_0$. Then the two limit gradient vectors given by (\ref{limit}) are already linearly independent. There are nothing to do further.

Case 2. Assume that $\ell_0=m_0$.   Then we take  a monomial function $\rho(t)=c t^{o_g-o_f},\,c\in \mathbb C$ and replace 
$\partial g(\mathbf z(t))$ by
\[
\mathbf v^{(1)}(t)=\overline{\partial g}(\mathbf z(t))-\rho(t)\overline{\partial f}(\mathbf z(t)). 
\]
We choose a constant $c$  so that 
$\ord\, v_{m_0}^{(1)}(t)>o_g$.  We put   $\ord\,v_{m_0}^{(1)}(t)=\infty$ if $v_{m_0}^{(1)}(t)\equiv 0$. Here 
$  v_{m_0}^{(1)}(t)$ is the $m_0$-th component of $\mathbf v^{(1)}(t)$.
Note that the two dimensional  complex subspace $W=\left<\overline{\partial f}(\mathbf z(t)), \overline{\partial g}(\mathbf z(t))\right>$ generated by 
$\{\overline{\partial f}(\mathbf z(t)), \overline{\partial g}(\mathbf z(t))\}$ is the same with subspace
$\left<\overline{\partial f}(\mathbf z(t)), \mathbf v^{(1)}(t)\right>$
 generated by $\{\overline{\partial f}(\mathbf z(t)), \mathbf v^{(1)}(t)\}$.
Thus their complex orthogonal subspaces are also equal.
We continue this operation
\[
\overline{\partial g}\to \mathbf v^{(1)}\to \dots \to \mathbf v^{(k)}
\]
 until the leading index of $\mathbf v^{(k)}$ changes. % $m_0$.
 Note that the k-times operation $\partial g(\mathbf z(t))\to \mathbf v^{(k)}(t)$ is given as
 \[
  \mathbf v^{(k)}(t)=\overline{\partial g}(\mathbf z(t))-\rho_k(t)\overline{\partial f}(\mathbf z(t))
 \]
 where $\rho_k(t)$ is a polynomial of variable $t$ whose lowest degree is $o_g-o_f$.
 By (\ref{non-degenerate}),  we may assume that ${\partial g_P}/{\partial z_a}(\mathbf w)\ne 0$.
 Note that $\ord\,  v^{(\nu)}_{m_0}(t)$ is strictly increasing as long as $\nu\le k-1$
 and $\ord\,\mathbf v^{(\nu)}(t)=\ord\, v^{(\nu)}_{m_0}(t)$. Let us look the components $a, b$ which is given by
 \[
v_{\tau}^{(k)}(t)=\overline{\frac{\partial g}{\partial z_\tau}}(\mathbf z(t))-\rho_k(t)\overline{\frac{\partial f}{\partial z_\tau}}(\mathbf z(t)),\,\,\tau=a,b.
 \]
 \begin{Assertion}One of the following inequalities holds.
  \[
  \ord\,v_{a}^{(k)}(t)\le d(P;g)-p_a,\quad \text{or}\quad\ord\,  v_{b}^{(k)}(t)\le d(P;g)-p_b.
  \]
 \end{Assertion}
 \begin{proof}
 Assume that $\ord\, v_a^{(k)}>d(P;g)-p_a$. We will show that $\ord\,  v_{b}^{(k)}(t)\le d(P;g)-p_b$.
 As the first term of $\rho_k(t) \overline{\frac{\partial f}{\partial z_a}}(\mathbf z(t))$  kill the first term of $\overline{\frac{\partial g}{\partial z_a}}(t)$,
 the order of 
 $\rho_k(t)\overline{\frac{\partial f}{\partial z_a}}(t)$ is equal to $d(P;g)-p_a$.
 There are two cases to be considered.
 
 (A) $\frac{\partial f_P}{\partial z_a}(\mathbf w)\ne 0$ or  (B) $\frac{\partial f_P}{\partial z_a}(\mathbf w)=0$.
 
 Assume the case (A). Then  we have $\ord\,\rho_k(t)=d(P;g)-d(P;f)=o_g-o_f$ and which implies 
 that 
 \[
 \ord\,\rho_k(t)\overline{\frac{\partial f_P}{\partial z_b}}(\mathbf z(t))\ge (d(P;f)-p_b)+(d(P;g)-d(P;f))=d(P;g)-p_b. 
 \]
 Then putting $\la$  be the coefficient of $t^{d(P;g)-d(P;f)}$ in  $\rho_k(t)$, we have
 \[\begin{split}
 0\ne 
& \det
\left(
\begin{matrix}
\overline{\frac{\partial f_P}{\partial z_a}}(\mathbf w)&\overline{\frac{\partial f_P}{\partial z_b}}(\mathbf w)\\
\overline{\frac{\partial g_P}{\partial z_a}}(\mathbf w)&\overline{\frac{\partial g_P}{\partial z_b}}(\mathbf w)\\
\end{matrix}
\right)\\
&=
\det
\left(
\begin{matrix}
\overline{\frac{\partial f_P}{\partial z_a}}(\mathbf w)&\overline{\frac{\partial f_P}{\partial z_b}}(\mathbf w)\\
\overline{\frac{\partial g_P}{\partial z_a}}(\mathbf w)+\la \overline{\frac{\partial f_P}{\partial z_a}}(\mathbf w)
&\overline{\frac{\partial g_P}{\partial z_b}}(\mathbf w)+\la \overline{\frac{\partial f_P}{\partial z_b}}(\mathbf w)
\end{matrix}
\right)
\\
&=
\det
\left(
\begin{matrix}
\overline{\frac{\partial f_P}{\partial z_a}}(\mathbf w)&\overline{\frac{\partial f_P}{\partial z_b}}(\mathbf w)\\
0
&\overline{\frac{\partial g_P}{\partial z_b}}(\mathbf w)+\la \overline{\frac{\partial f_P}{\partial z_b}}(\mathbf w)
\end{matrix}
\right)
\end{split}
 \]
 and thus
 $\overline{\frac{\partial g_P}{\partial z_b}}(\mathbf w)+\la \overline{\frac{\partial f_P}{\partial z_b}}(\mathbf w)\ne 0$ by (\ref{non-degenerate}),
 $\ord\,  v^{(k)}_b=d(P;g)-p_b$.
  
 Consider the case (B) now.  Then $\ord\,\rho_k(t)<d(P;g)-d(P;f)$. By (\ref{non-degenerate}), 
 $\frac{\partial f_P}{\partial z_b}(\mathbf w)\ne 0$ and $\ord\, v^{(k)}(t)_b<d(P;g)-p_b$.
 Thus in both cases, under the above operation, $\ord\,  v^{(k)}(t)\le d-p_m$. 
 \end{proof}
 The above argument implies  that  the number $k$ of operations is bounded by $k\le d(P;g)-p_m-o_g$.
 % and  after a finite number of operation, say $\ell$  times,
 %the leading indices of $\mathbf v^{(\ell)}(t)$ changes.
  At the last operation, the leading index of $\mathbf v^{(k)}(t)$ is different from $m_0$ and  the limit vector of $\mathbf v^{(k)}(t)$ and $\overline{\partial f}(\mathbf z(t))$
 are linearly independent and they are in the subspace $\mathbb C^J$.
 %This implies the following.
As \nl
$T_{\mathbf z(t)} V(f,\mathbf z(t))\cap T_{\mathbf z(t)} V(g,\mathbf z(t))$  is the complex orthonormal subspace of the two dimensional subspace
 $\left <\overline{\partial f}(\mathbf z(t)), \overline{\partial g}(\mathbf z(t))\right>_{\mathbb C}$
 and it is equal to  the complex  orthonormal subspace of 
 $\left<\overline{\partial f}(\mathbf z(t)),\mathbf v^{(k)}(t)\right>_{\mathbb C}$, the limit of $T_{\mathbf z(t)}V( f,\mathbf z(t))\cap T_{\mathbf z(t)} V(g,\mathbf z(t))$
 includes the vanishing subspace $\mathbb C^I$.
 As $T_{\mathbf z(t)} V(H,\mathbf z(t))$ includes  
 $T_{\mathbf z(t)} V(f,\mathbf z(t))\cap T_{\mathbf z(t)} V(g,\mathbf z(t))$ as a subspace,
 $\lim_{t\to 0}T_{\mathbf z(t)} V(H,\mathbf z(t))\supset \mathbb C^I$.
 Thus
 the proof of case (c-1) is done.
 \nl
 Now we consider the case (c-2):  $f_P(\mathbf w)=0,g_P(\mathbf w)\ne 0$. %(The case (c-2)' is similar.)
 Consider  the hypersurface $V:=\{(\mathbf z_J\in \mathbb C^{*J}\,|\, f_P(\mathbf z)=0,\, \mathbf z_I=\mathbf a\}$.
 By the locally tameness assumption, $V$ is a non-singular hypersurface in $\mathbb C^{*I}$. Let us consider
 the restriction  $g_P:V\to \mathbb C$.
 As $g_P$ is a weighted homogeneous polynomial function of weight $P_J=(p_{1},\dots, p_m)$ and $V$ is invariant under the associated $\mathbb C^*$ action on $\mathbb C^J$,
 $g_P$ has no non-zero critical value on $V$. Namely 
 $\bar \partial g_P(\mathbf w)$ is non-zero and linearly independent with $\bar \partial f_P(\mathbf w)$.
 Thus there is a pair $a\le b\le m$ which satisfies $(\ref{non-degenerate})$. Thus the rest of the argument is the exact same as above and 
 $\lim_{s\to 0}T_{\mathbf z(s)}V(H,\mathbf z(s))\supset \mathbb C^I$. The case (c-2)':  $f_P(\mathbf w)\ne 0$ and $g_P(\mathbf w)=0$ is treated similarly.
 \nl
 Now we consider the case (c-3): $f_P(\mathbf w), g_P(\mathbf w)\ne 0$.
 In this case, we need the assumption that $f,g$ satisfies the  tame Newton multiplicity condition.
 Let $\deg_Pf=d_f$ and $\deg_P g=d_g$.  Put $d_r:=d_f+d_g$ and $d_p=d_f-d_g$. Then the  tame Newton multiplicity condition implies  $d_p\ne 0$. The mixed function
 $H_P=f_P\bar g_P:\mathbb C^{J}\to \mathbb C$  is a  strongly mixed weighted homogeneous polynomial which satisfies
 $H(\rho e^{i\theta}\circ\mathbf z)=\rho^{d_r} e^{d_p\theta} H(\mathbf z)$ and thus $0$ is the only critical value.
 Thus $\bar \partial \Re H_P(\mathbf w),\bar \partial \Im H_P(\mathbf w)$ are linearly independent over $\mathbb R$ and we can proceed the same argument as above replacing $\bar\partial f,\bar \partial g$ to 
 $\bar \partial \Re H(\mathbf z(t)),\bar \partial \Im H(\mathbf z(t))$ to conclude the real two dimensional subspace
 $\langle \bar \partial \Re H(\mathbf z(t)),\bar \partial \Im H(\mathbf z(t))\rangle_{\mathbb R}$ has a limit which is a real 2-dimensional subspace
 of $\mathbb C^J$.  Thus $\lim_{t\to 0} T_{\mathbf z(t)} V(H,\mathbf z(t))\supset \mathbb C^I$. 
 See  the proof of Theorem 20, \cite{OkaAf} for further detail.
 \end{proof}
 Now the proof of Theorem \ref{af-regularity} is completed.

\end{proof}
 By Proposition 11, \cite{OkaAf}, we get the transversality assertion:
\begin{Corollary} \label{nearby-transversality}Let $\{f,g\}$ be as in Theorem \ref{af-regularity}. Then there exists a positive number $r_0$ such that for any $r_1,\,0<r_1\le r_0$, there exists a positive number $\de(r_1)$ so that 
for any $\eta\ne 0$  with $|\eta|\le \de(r_1)$ and  and  any $\rho$, $r_1\le \rho\le r_0$, the nearby fiber $H\inv(\eta)$ is non-singular in $B_{r_0}^{2n}$ and intersects transversely with the sphere $S_{\rho}^{2n-1}$.% for any $r_1\le \rho\le r_0$.
\end{Corollary}
%\begin{Conjecture}[Conjecture-Problem]Assume that $\{f,g\}$ is a weak locally tame non-degenerate pair and $H$
%satisfies the toric multiplicity condition. Then $H:=f\bar g$ has an isolated critical value. Does $H$ satisfy $a_f$-regularity  with respect to $\mathcal S$ or some other  stratification $\mathcal S'$ of $V(H)$? 
%\end{Conjecture}
%%%%%%%%%%%%%%%%%%%%%%%%%%%

%\begin{enum
\subsection{Existence of a tubular Milnor fibration}
By Lemma \ref{smoothness}, Theorem\ref{af-regularity} and Corollary \ref{nearby-transversality}, we apply Ehresmann's fibration theorem  ({\cite{Wolf1}) to obtain:
\begin{Theorem}\label{tubular-Milnor}
Assume that $\{f,g\}$ satisfies the same assumption as in Theorem \ref{af-regularity}.
 %is a locally tame  non-degenerate complete intersection pair which satisfies the toric multiplicity condition .
Then there exists a  positive number $\eps$ and  a sufficiently small $\de\ll\eps$ such that 
\[
H=f\bar g:E(\eps;\de)^* \to D_\de^*
\]
is a locally trivial fibration where $E(\eps,\de)^*\,:=\{(\mathbf z)\,|\, 0\ne |H(\mathbf z)|\le \de,\,\|\mathbf z\|\le \eps\}$
and $D_\de^*:=\{\zeta\in \mathbb C\,|\, 0\ne |\zeta|\le \de\}$,
\end{Theorem}
By Corollary \ref{nearby-transversality}, the fibration does not depend on the choice of $\eps$ and $\de$. 
For a disjoint locally tame non-degenerate pair $\{f,g\}$,  applying  the argument of  2 of  Example \ref{Ex5}, we have:
\begin{Corollary} Asssume that $\{f,g\}$ is a disjoint locally tame  non-degenerate complete intersection pair.  Fix an admissible toric modification $\hat \pi:X\to\mathbb C^n$.
Take $m>0$ large so that $\{f_m,g\}$  and $\{f,g_m\}$ satisfy the toric multiplicity condition with $\hat \pi$. Then changing the coefficients of $f_m$ or $g_m$ slightly if necessary, $\{f_m,g\}$  and $\{f,g_m\}$ are locally tame  non-degenerate complete intersections respectively
for which $\hat \pi:X\to\mathbb C^n$ satisfies the toric multiplicity condition.
 Thus $H_m:=f_m\bar g$ and $K:=f\bar g_m$  have  tubular Milnor fibrations.
\end{Corollary}
For the definition $f_m$, see Example \ref{Ex5}.
%See Example \ref{Ex5}, 2.
%\end{enumerate}
%\end{proof}
\section{Spherical Milnor fibration}
In this section, we study the existence of the spherical Milnor fibration.
For a fixed small $r>0$, we consider the mapping
$\vphi:S_r^{2n-1}\setminus K\to S^1$ where $K=V(H)\cap S_r^{2n-1}$ and $\vphi(\mathbf z)=H(\mathbf z)/|H(\mathbf z)|$.
\begin{Lemma} [Lemma 10,\cite{fg-bar}]\label{spherical}We assume $\{f,g\}$ is  a  weak locally tame non-degenerate complete intersection pair 
satisfying the assumption in Theorem \ref{af-regularity}. %satisfying  the  tame Newton multiplicity condition or the disjoint vanishing coordinate subspaces. 
%In the latter case, we assume that $\{f,g\}$ satisfies also a toric multiplicity condition. 
Then there exists a positive number $r_3$ so that $\vphi: S_r^{2n-1}\setminus K\to S^1$ has no critical points 
 for  any $r,\,0<r\le r_3$.
\end{Lemma}
In the proof of Lemma \ref{positive} below, we will simultaneously reprove Lemma \ref{spherical}.
Using  Lemma\ref{spherical} and the transversality property of the  fibers  $H\inv(\eta),\ 0\ne |\eta|\le \de$ and the sphere $S_r^{2n-1}$ (Corollary \ref{nearby-transversality}), we obtain the following.
\begin{Theorem}\label{SphericalMilnor}Assume $\{f,g\}$ is  a weak  locally tame non-degenerate complete intersection pair  as in Theorem \ref{af-regularity}. For a sufficiently small $r$,
$\vphi:S_r^{2n-1}\setminus K\to S^1$ is a locally trivial fibration.
\end{Theorem}
\begin{proof} Consider the neighborhood of $K$ defined by 
$N(K):=\{\mathbf z\in S_r^{2n-1}\setminus K\,|\, |H(\mathbf z)|\le \de\}$.
Corollary \ref{nearby-transversality} says that  three vectors $\mathbf z,\mathbf v_1(\mathbf z),\mathbf v_2(\mathbf z)$ are linearly independent over $\mathbb R$ on $ N(K)$. For the definition of $\mathbf v_1,\mathbf v_2$, see the next section. 
Construct a vector filed $\mathcal V$ on $S_r^{2n-1}\setminus K$ such that 
$\Re(\mathcal V(\mathbf z),\mathbf v_2(\mathbf z))=1$ and furthermore if $\mathbf z\in N(K)$, it also satisfies $\Re(\mathcal V(\mathbf z),\mathbf  v_1(\mathbf z))=0$.
Then along the integral curves of $\mathcal V$, the argument of $H(\mathbf z)$ is monotonic  increase and   the absolute value of $H$ is constant when it enters  in 
the neighborhood $N(K)$.
Thus the integral curves exists for any time interval. For the local triviality, we use the integration of $\mathcal V$.
\end{proof}
\subsection{Equivalence of tubular  and spherical Milnor fibrations}
In this section, we consider the equivalence problem of two Milnor fibrations.
Let us recall  two vector fields on the complement of $V(H)$ which are defined as follows (\cite{OkaMix}).

\begin{eqnarray*}
\mathbf v_1&=&\overline{\partial \log H}+\bar\partial \log H
= \frac{\overline{\partial f}}{\bar f}+ \frac{\overline{\partial g}}{\bar g}\\
\mathbf v_2&=& i(\overline{\partial\log H}-\bar\partial \log H)
                    =i \left( \frac{\overline{\partial f}}{\bar f}- \frac{\overline{\partial g}}{\bar g}   \right).
\end{eqnarray*}
$\mathbf v_1,\mathbf v_2$ are real orthogonal.
Let $\mathbf z(t)$ be a real analytic curve in $\mathbb C^n\setminus V(H)$. Then we have
\begin{eqnarray*}
&&\frac d{dt}\log H(\mathbf z(t))\\
&&=\frac{1}{f(\mathbf z(t))} \sum_{i=1}^n \frac{\partial f}{\partial z_i}(\mathbf z(t))\frac{dz_i(t)}{dt}+
\frac{1}{\bar g(\mathbf z(t))}\sum_{i=1}^n\overline{ \frac{\partial g}{\partial z_i}(\mathbf z(t))}\overline{\frac{dz_i(t)}{dt}}\\
&&=\frac 12 \left(\frac{d\mathbf z(t)}{dt}, \mathbf v_1(\mathbf z(t))-i\mathbf v_2(\mathbf z(t)) \right)+\frac 12\overline{ \left(\frac{d\mathbf z(t)}{dt}, \mathbf v_1(\mathbf z(t))+i \mathbf v_2 (\mathbf z(t))\right)}\\
&&=\Re\left (\frac{d\mathbf z(t)}{dt}, \mathbf v_1(\mathbf z(t))\right)+i\Re\left (\frac{d\mathbf z(t)}{dt}, \mathbf v_2(\mathbf z(t))\right).
\end{eqnarray*}
Thus we $\mathbf v_1(\mathbf z)$ and $\mathbf v_2(\mathbf z)$ are gradient vectors of $\Re\log H(\mathbf z)=\log |H(\mathbf z)|$ and $\Im\log H(\mathbf z)=i\arg H(\mathbf z)$. They are defined on $\mathbb C^n\setminus V(H)$.
A key lemma is the following.
\begin{Lemma}\label{positive}
Assume that $\{f,g\}$  is as in Theorem \ref{af-regularity}.
%is a locally tame non-degenerate complete intersection pair.
There exists a positive number $r_0$ such that  for any  $\mathbf z\in B_{r_0}^{2n}\setminus V(H)$,
either three vectors 
$\mathbf z,\mathbf v_1(\mathbf z),\mathbf v_2(\mathbf z)$ are linearly independent over $\mathbb R$ or they are linearly dependent and 
  the relation takes the following form:
\[
\mathbf z=\lambda \mathbf v_1(\mathbf z)+\mu \mathbf v_2(\mathbf z),\,\lambda,\mu\in \mathbb R,
\]
where 
$\lambda$ is positive.
\end{Lemma}
\begin{proof}
Assume that there exists a real analytic curve $\mathbf z(t)$ in $\mathbb C^n\setminus V(H)$ and real valued rational functions
$\lambda(t),\mu(t)$ such that 
\begin{eqnarray}\label{eq1}
\mathbf z(t)=\lambda(t) \mathbf v_1(\mathbf z(t))+\mu(t) \mathbf v_2(\mathbf z(t))
\end{eqnarray}
and $\mathbf z(0)=\mathbf 0$.  %Note that $\la(t)\not\equiv 0$ by Lemma \ref{spherical}.
 If $\mu(t)\equiv 0$, the assertion follows from Corollary 3.4, \cite{Milnor}. Thus we may assume that $\mu(t)\not \equiv 0$. 
Let $I=\{j\,|\, z_j(t)\not\equiv 0\}$. As $f^I,g^I$ satisfies the same assumption, we assume for simplicity that $I=\{1,\dots,n\}$.
Thus $\mathbf z(t)\in \mathbb C^{*n}$ for $t\ne 0$. Consider their Taylor or Laurent expansions
\begin{eqnarray*}
f(\mathbf z(t))&=&\ga t^{m_f}+\text{(higher terms)},\, \ga\in \mathbb C^*\\
g(\mathbf z(t))&=&\be t^{m_g}+\text{(higher terms)},\, \be\in \mathbb C^*\\
z_i(t)&=&a_it^{p_i}+\text{(higher terms)},\, a_i\in \mathbb C^*\\
\la(t)&=&\la_0t^{\nu_1}+\text{(higher terms)},\, \,\la_0\in \mathbb R\\
\mu(t)&=&\mu_0t^{\nu_2}+\text{(higher terms)},\, \mu_0\in \mathbb R^*.
\end{eqnarray*}
In the proof, we reprove Lemma \ref{spherical}. Thus $\la_0=0$ only if $\la(t)\equiv 0$ and in  that case, we understand $\nu_1=+\infty$. 
If this is the case, $\mathbf z(t)$ is a critical point of $\vphi: S_\tau^{2n-1}\setminus K_\tau$ where $\tau=\|\mathbf z(t)\|$ and $K_\tau$ is the link of $H\inv(0)$ in this sphere.
Put  $\ell:=\min\{d(f;P)-m_f,\,d(P;g)-m_g\}$ and let us define
\begin{eqnarray*}
%&&\ell:=\min\{d(f;P)-m_f,\,d(P;g)-m_g\},\\
&&\eps_f= \begin{cases} &1,\,\text{if}\, \,d(P;f)-m_f=\ell\\
                                      &0,\, \text{if}\,\,d(P;f)-m_f> \ell
                                      \end{cases}\\
  & &\eps_g=\begin{cases} &1,\,\text{if}\,\,d(P;g)-m_g=\ell\\
                                    &0,\, \text{if}\,\,d(P;g)-m_g> \ell\\         
                                   \end{cases}                    
\end{eqnarray*}
As $\ord\,f(\mathbf z(t))\ge d(P;f),\,\ord\,g(\mathbf z(t))\ge d(P;g)$, 
we have  that $\ell\le 0$.
Put $\mathbf v_k(\mathbf z(t))=(v_k^1(t),\dots, v_k^n(t))$ for $k=1,2$.
 Observe that 
\begin{eqnarray}%%z_i(t)&=\\
&v_1^j(t)=\left(\overline{\dfrac{f_{P,j}(\mathbf a)}{\bar{\ga}}}\eps_f+\overline{\dfrac{g_{P,i}(\mathbf a)}{\bar{\be}}}\eps_g
\right)t^{\ell-p_j}+\dots\label{eq2}\\ %\text{(higher terms)}\\
&v_2^j(t)= i\left(\overline{\dfrac{f_{P,j}(\mathbf a)}{\bar{\ga}}}\eps_f-\overline{\dfrac{g_{P,j}(\mathbf a)}{\bar{\be}}}\eps_g
\right)t^{\ell-p_j}+\dots\label{eq3}
\end{eqnarray}
%Here $v_k^j(\mathbf z(t))$ is the $j$-th component  of $\mathbf v_k(\mathbf z(t))$, for $k=1,2$.  
Put 
\[\begin{split}
&P(\mathbf a):=(p_1a_1,\dots,p_na_n),\,\,
p_{min}=\min\{p_j\,|\, j\in I\},\\
&J=\{j\,|\, p_j=p_{min}\}, \,\nu_0=\min\{\nu_1,\nu_2\}
\end{split}
\]
 and put $\de_i=1$ or $0$ according to $\nu_i=\nu_0$ or $\nu_i>\nu_0$ respectively for $i=1,2$.
By (\ref{eq1}), we get
\begin{eqnarray*}\label{assumption}
a_jt^{p_j}&+&\dots\\
%\left(
&=&\la_0\left(\overline{\frac{f_{P,j}(\mathbf a)}{\bar{\ga}}}\eps_f+\overline{\frac{g_{P,i}(\mathbf a)}{\bar{\be}}}\eps_g
\right)t^{\nu_1+\ell-p_j}+\dots\\
&+&\mu_0 i\left(\overline{\frac{f_{P,j}(\mathbf a)}{\bar{\ga}}}\eps_f-\overline{\frac{g_{P,j}(\mathbf a)}{\bar{\be}}}\eps_g
\right)t^{\nu_2+\ell-p_j}+\dots\\
&=&e_jt^{\nu_0+\ell-p_j}
+\dots
\end{eqnarray*}
where 
\[e_j
=\left\{\la_0\de_1\left(\overline{\frac{f_{P,j}(\mathbf a)}{\bar{\ga}}}\eps_f+\overline{\frac{g_{P,i}(\mathbf a)}{\bar{\be}}}\eps_g
\right)+
\mu_0\de_2  i\left(\overline{\frac{f_{P,j}(\mathbf a)}{\bar{\ga}}}\eps_f-\overline{\frac{g_{P,j}(\mathbf a)}{\bar{\be}}}\eps_g
\right)
\right\}.
\]

If $\nu_0+\ell-2p_{min}>0$, we get a contradiction  $a_j=0,\, j\in J$. Thus $\nu_0+\ell-2p_{min}\le 0$.
%So only possible case is $\nu_0+\ell-2p_{min}=0$. We will show this in the proof.
Consider the vectors
\begin{eqnarray*}
\mathbf v_1^{(0)}&=&(w_1^{1},\dots, w_1^n),\,\,
 w_1^{j}=\overline{\frac{f_{P,j}(\mathbf a)}{\overline{\ga}}}\eps_f+\overline{\frac{g_{P,j}(\mathbf a)}{\overline{\be}}}\eps_g\\
 \mathbf v_2^{(0)}&=&(w_2^1,\dots, w_2^n),\,\,
 w_2^{j}=i\left( \overline{\frac{f_{P,j}(\mathbf a)}{\overline{\ga}}}\eps_f-\overline{\frac{g_{P,j}(\mathbf a)}{  \overline{\be}}} \eps_g\right)
\end{eqnarray*}
%We will show that $\nu_0+\ell-2p_{min}=0$.
Assume that $\nu_0+\ell-2p_{min}<0$. 
By (\ref{assumption}), we get
\begin{eqnarray}\label{vanish}
\la_0\de_1w_1^j+\mu_0\de_2w_2^j=0,\,\,j=1,\dots,n.
%&&\l\ell_0\de_1\left(\overline{\frac{f_{P,j}(\mathbf a)}{\bar{\ga}}}\eps_f+\overline{\frac{g_{P,i}(\mathbf a)}{\bar{\be}}}\eps_g\right)+
%\mu_0\de_2 i\left(\overline{\frac{f_{P,j}(\mathbf a)}{\bar{\ga}}}\eps_f-\overline{\frac{g_{P,j}(\mathbf a)}{\bar{\be}}}\eps_g
%\right)=0\\
%&&\qquad j=1,\dots,n.\notag
\end{eqnarray}
If $\ell<0$, $\eps_f f_P(\mathbf a)=0$  and $\eps_g g_P(\mathbf a)=0$.
The above equality gives a contradiction to the non-degeneracy condition either  for 
$V(f)$ if $\eps_f=1,\eps_g=0$, or  %the non-degeneracy condition 
for  $V(g)$ if $\eps_f=0,\eps_g=1$  or 
for the intersection variety $V(f,g)$ if $\eps_f=\eps_g=1$.

Assume $\ell=0$. Then  $\nu_0<2p_{min}$, $\ga=f_P(\mathbf a)$ and $\be=g_P(\mathbf a)$.
We consider the equality
%\begin{eqnarray}
%\begin{split}\left (\frac{d\mathbf z(t)}{dt},\mathbf z(t)\right)&=\left(\frac{d\mathbf z(t)}{dt},\la(t)v_1(\mathbf z(t))+\mu(t)\mathbf v_1(\mathbf z(t))\right).
%&=(\frac{d\mathbf z(t)}{dt},\la(t)v_1(\mathbf z(t)))+
%2\Re(\frac{d\mathbf z(t)}{dt},\la(t)v_2(\mathbf z(t))).
%\end{split}
\begin{eqnarray}\label{comparison0}
\sum_{j\in J}\frac{dz_j(t)}{dt}z_j(t)=\sum_{j\in J}\frac{dz_j(t)}{dt}\left({\overline{\la(t)}\overline{v_1^{j}}(t)+\overline{\mu(t)}\overline {v_2^{j}}(t)}\right).
\end{eqnarray}
The left hand side has order $2p_{min}-1$ as
\[
\sum_{j\in J}\frac{z_j(t)}{dt}z_j(t)=\sum_{j\in J}^n p_j|a_j|^2 t^{2p_{min}-1}+\dots.
%\left(\frac{d\mathbf z(t)}{dt},\mathbf z(t)\right)= \sum_{j\in J}^n p_j|a_j|^2 t^{2p_{min}-1}+\dots.
\]
%On the other hand $\ord\,(\la(t)\mathbf v_1(\mathbf z(t))+\mu(t)\mathbf v_2(\mathbf z(t)))<\nu_0-p_{min}<p_{min}$ and 
Using (\ref{vanish}) and Euler equality, we see that the leading term of the right hand is $t^{\nu_0-1}$  which has the coefficient
%\begin{eqnarray*}
%\left(\frac{d\mathbf z(t)}{dt},\la(t)\mathbf v_1(\mathbf z(t))+\mu(t)\mathbf v_2(\mathbf z(t))\right)
%\end{eqnarray*}
\[\la_0\de_1(d(P;f)+d(P;g))+i\mu_0\de_2(d(P;f)-d(P;g))\ne 0.
\]
The coefficient is non-zero.
( If $\de_1=0$, we use the Newton multiplicity condition to see the imaginary part is non-zero.) Thus the order is strictly smaller than $2p_{min}-1$, which is a contradiction.
Thus the case  $\nu_0+\ell-2p_{min}<0$ does not occur.
%This is again a contradiction. 
Thus the following  equality holds:
\[
\nu_0+\ell-2p_{min}=0.
\]
(\ref{assumption}) implies the following equality.
\begin{eqnarray} \label{comparison}
%\l\la_0\de_1\left(\overline{\frac{f_{P,j}(\mathbf a)}{\bar{\ga}}}\eps_f+\overline{\frac{g_{P,i}(\mathbf a)}{\bar{\be}}}\eps_g\right)&+&
%\mu_0\de_2 i\left(\overline{\frac{f_{P,j}(\mathbf a)}{\bar{\ga}}}\eps_f-\overline{\frac{g_{P,j}(\mathbf a)}{\bar{\be}}}\eps_g
%\right)\\
\la_0\de_1 w_1^j+\mu_0\de_2 w_2^j
&=&\begin{cases}
a_j\, \, &j\in J\\
0,\,\,&j\notin J.
\end{cases}\notag
\end{eqnarray}
 We consider the equality (\ref{comparison0}) again.
%\begin{eqnarray}\label{comparison}
%\sum_{j\in J}\frac{z_j(t)}{dt}z_j(t)=\sum_{j\in J}\frac{dz_j(t)}{dt}\left({\overline{\la(t)}\overline{v_1^{j}}(t)+\overline{\mu(t)}\overline {v_2^{j}}(t)}\right).
%\end{eqnarray}

The left side of (\ref{assumption}) has order $2p_{min}-1$ with the coefficient $\sum_{j\in J}p_{min}|a_j|^2>0$.
The right side has order $2p_{min-1}$ and the coefficient is given through Euler equality  as
\begin{multline*}
\la_0\de_1
\left( \frac{d(P;f)f_P(\mathbf a)}{\ga}\eps_f+\frac{d(P;g)g_P(\mathbf a)}{\be}\eps_g\right)\\
+\mu_0\de_2 i
\left(\frac{d(P;f)f_P(\mathbf a)}{\ga}\eps_f-\frac{d(P;g)g_P(\mathbf a)}{\be}\eps_g\right)
\end{multline*}
If $\ell<0$, $f_P(\mathbf a)\eps_f=g_P(\mathbf a)\eps_g=0$ and the above coefficient is zero. Thus we get a contradiction.
Thus the only possible case is $\ell=0$ and therefore
\[f_P(\mathbf a),\,g_P(\mathbf a)\ne 0,\, \nu_0=2p_{min}.
\]
We observe also $\de_1\ne 0$, as otherwise the coefficient is purely imaginary.
Thus we should have 
\[\ell=0, \,\nu_1\le \nu_2,\,\ga=f_P(\mathbf a),\, \be=g_P(\mathbf a).
\]
The leading coefficients of (\ref{comparison0}) gives the equality:
\begin{multline*}
\sum_{j\in J}p_{min}|a_j|^2
=\la_0\left (d(P;f)+d(P;g)\right )+i\de_2\mu_0\left(d(P;f)-d(P;g)\right).
\end{multline*}
Thus  taking the real part of this equality, we conclude that $\la_0>0$.
This also proves  $\la(t)\equiv 0$ does not occur as  $\la_0\ne 0$. This gives another proof of Corollary \ref{nearby-transversality}.
\end{proof}
%\end{comment}
%\end{proof}
%%%%%%%%%%%%%%%%%%%%%%
Now we are ready to prove the equivalence theorem.
\begin{Theorem} Assume that $\{f,g\}$ is a locally tame non-degenerate complete intersection pair as in Theorem \ref{af-regularity}.
Consider the tubular and spherical Milnor fibrations
\[
\begin{split}
&H:\partial E(r,\de)^*\to S_\de^1\\
&\vphi: S_r^{2n-1}\setminus K\to S^1.
\end{split}
\]
These two fibrations are equivalent.
Here we use the notations
\[
\partial E(r,\de)^*:=\{\mathbf z\in B_r^{2n}\,|\, |H(\mathbf z)|=\de\},\,
K=S_r^{2n-1}\cap V(H).
\]
\end{Theorem}
\begin{proof}Let $\de$ be sufficiently small and put 
$\partial N(K):=\{\mathbf z\in S_r^{2n-1}\,|\,|H(\mathbf z)|<\de\}$. By the transversality, $N(K)$ is contractible to $K$ and $N(K)\setminus K$ is diffeomorphic to $\partial N(K)\times (0,1)$ %where $\Int D_\de=\{\eta\in \mathbb C\,|\, |\eta|<\de\}$
and $\vphi:S_r^{2n-1}\setminus N(K)\to S^1$ is equivalent to the spherical fibration $\vphi:S_r^{2n-1}\setminus K\to S^1$.
Note that  vectors $\mathbf v_1,\mathbf v_2$ are real orthogonal.
Take a locally finite open covering $\mathcal U=\{U_\al,\,\al\in A\}\cup\{V_\be,\,\be\in B\}$ of $B_r^{2n}\cap\{\mathbf z\,|\, |H(\mathbf z)|\ge \de\}$
as follows. Each $U_\al, V_\be$ are open disk  with center $p_\al, p_\be$. Secondly 
 in each $U_\al$, $\{\mathbf z,\mathbf v_1(\mathbf z), \mathbf v_2(\mathbf z)\}$  are linearly independent over $\mathbb R$, while in $V_\be$,
 $p_\be$ can be written as
 \[
 p_\be=\la \mathbf v_1(p_\be)+\mu \mathbf v_2(p_\be),\quad \la>0
 \]
 and we take the radius of $V_\be$  is small enough so that 
%they may  not be linearly independent but   the relation satisfies Lemma \ref{positive} i.e.,
 $\Re(\mathbf z,\mathbf v_1(\mathbf z))>0$ for any $\mathbf z\in V_\be$. (There might exist a point $\mathbf z\in V_\be$ where $\mathbf z, \mathbf v_1(\mathbf z),\mathbf v_2(\mathbf z)$ are linearly independent.)
Construct a vector field  $\mathbf w_\al$ on $U_\al$  
%and $\mathbf w_\be$ on $V_\be $ 
so that
\[\begin{split}
&\Re(\mathbf w_\al(\mathbf z),\mathbf v_2(\mathbf z))=0,\Re(\mathbf w_\al(\mathbf z),\mathbf v_1(\mathbf z))=1,\,
\Re(\mathbf w_\al(\mathbf z),\mathbf z)=1,\,\,\forall\mathbf z\in U_\al.
%&\Re(\mathbf w_\be(\mathbf z),\mathbf v_2(\mathbf z))=0,\Re(\mathbf w_\be(\mathbf z),\mathbf v_1(\mathbf z))>0,\, \mathbf z\in V_\be.
\end{split}
\]
On $V_\be$, we simply take $\mathbf w_\be(\mathbf z)=\mathbf v_1(\mathbf z)$. As $V_\be$ is small enough,   $\Re(\mathbf w_\be(\mathbf z),\mathbf z)>0$
and $\Re(\mathbf w_\be(\mathbf z),\mathbf v_2(\mathbf z))=0$
for $\forall \mathbf z\in V_\be$.
Then glue together these vectors using a partition of unity to get a vector field $\mathcal X(\mathbf z)$ on $B_r^{2n}\cap\{\mathbf z\,|\, |H(\mathbf z)|\ge \de\}$.
Note that for any $\mathbf z\in B_r^{2n}\cap\{\mathbf z\,|\, |H(\mathbf z)|\ge \de\}$, $\Re(\mathcal X(\mathbf z),\mathbf v_2(\mathbf z))=0$ and 
$\Re(\mathcal X(\mathbf z),\mathbf z)>0$  and $\Re(\mathcal X(\mathbf z),\mathbf v_1(\mathbf z))>0$ by the construction.
Such a vector field is called a Milnor vector field in \cite{ART1,ART2}. Along any integration curve $\mathbf z(t)$ starting a point
$p\in \partial E(r,\de)^*$, $\arg H(\mathbf z(t))$ is constant  and $|H(\mathbf z(t))|,\,\|\mathbf z(t)\|$ are strictly increasing. This curve  arrives at
$\mathbf z(s(p))\in S_r^{2n-1}\setminus N(K)$ for a finite time $s(p)>0$.
Using this integration and the correspondence $p\mapsto \mathbf z(s(p))$, we can construct a diffeomorphism  %$\psi$ 
%which gives the equivalence of two Milnor fibrations
$\psi: \partial E(r,\de)^*\to S_r^{2n-1}-N(K)$      %, $N(K):=\{\mathbf z\in S_r^{2n-1}\,|\, |H(\mathbf z)|< \de\}$ 
and $\psi$ gives the commutative diagram:
\[
\begin{matrix}
\partial E(r,\de)^*&\mapright{\psi}& S_r^{2n-1}-N(K)\\
\mapdown{H}&&\mapdown {\vphi}\\
S_\de^1&\mapright{\iota}&S^1
\end{matrix}
\]
where $\iota(\eta)=\eta/\de$. Thus $\psi$ gives an isomorphism of the two fibrations.
\end{proof}

%\subsection{Equivalence of tubular Milnor fibration and spherical Milnor fibration}
%We assume theat $H$ is  locall tane non-dgenerate 
%\input NonDeg3.bbl
\def\cprime{$'$} \def\cprime{$'$} \def\cprime{$'$} \def\cprime{$'$}
  \def\cprime{$'$} \def\cprime{$'$} \def\cprime{$'$} \def\cprime{$'$}


\begin{thebibliography}{10}
\bibitem{AC}
N. A'Campo.
\newblock {La fonction zeta d'une monodromie.}
\newblock {\em Commentarii Mathematici Helvetici}, 50, (1975),  233-248.

\bibitem{AT}
 R.N. Araujo dos Santos,  M. Tibar.
 Real map germs and higher open book structures.
 {\em Geom. Dedicata}, 147, (2010), 177-185.
\bibitem{SCT1}
 R.N. Araujo dos Santos, Y. Chen, M. Tibar.
Singular open book structures from real mappings, {\em Cent. Eur. J. Math.}, 11, (2013), no. 5, 817-828.

\bibitem {SCT2}
 R.N. Araujo dos Santos, Y. Chen,  M.Tibar
Real polynomial maps and singular open books at infinity, 
{\em Math. Scand. }, 118, (2016), no.1, 57-69.
\bibitem{ART1}
R.N. Araujo dos Santos, M. Ribeiro and M. Tibar.
Fibrations of highly singular map germs,
Bull. Sci. Math. 55, (2019), 92-111.
%ArXiv:1711.07544
\bibitem{ART2}
R.N. Araujo dos Santos, M. Ribeiro and M. Tibar.
Milnor-Hamm sphere fibrations and the equivalence problem,
%Milnor-Hamm sphere fibrations and the equivalence  problem,
arXiv:1810.05158
\bibitem{Chen}
Y.~ Chen.
\newblock Ensembles de bifurcation des polyn\^omes mixtes et poly\`edres de Newton, Th\`ese, {\em Universit\'e de Lille I,}
2012.
\bibitem{EO} C. Eyral and M. Oka.
 Whitney regularity and Thom condition for the family of non-isolated mixed singularities,
J. Math. Soc. Japan,  70, (2018), no.4, 1305-1336.

\bibitem {JM}
J. Fernandez de Bobadilla, A. Menegon Neto.
The boundary of the Milnor  fibre of complex and real analytic non-isolated singularities. 
{\em Geom Dedicata}, 173,  (2014), 143-162



\bibitem{Hamm1}
H.~Hamm.
\newblock Lokale topologische {E}igenschaften komplexer {R}\"aume.
\newblock {\em Math. Ann.},191, (1971), 235-252.

\bibitem{Hamm-Le1}
H.~A. Hamm and D.~T. L{\^e}.
\newblock Un th\'eor\`eme de {Z}ariski du type de {L}efschetz.
\newblock {\em Ann. Sci. \'Ecole Norm. Sup.}(4), (1973), 6:317-355.
\bibitem{Joita-Tibar}
C. Joita and M. Tibar.
Images of analytic map germs,
arXiv:1810.05158
\bibitem{Kouchnirenko}
{A. G. Kouchnirenko},
\newblock { Poly{\`e}dres de Newton et nombres de {Milnor}}.
\newblock{\em Invent. Math.}, 32, (1976),
{1-31}.
\bibitem{Le-Saito}
D.T. L\^e and K. Saito.
The local $\pi_1$ of the complement of a hypersurface with normal crossings in codimension 1 is abelian.
{\em Ark. Mat.}, 22,  (1984), no. 1, 1-24.

\bibitem{Milnor}
J.~Milnor.
\newblock { Singular points of complex hypersurfaces}.
\newblock {\em  Annals of Mathematics Studies,}  61. Princeton University Press,
  Princeton, N.J., 1968.

\bibitem{OkaPrincipal}
   M. Oka.
     {Principal zeta-function of nondegenerate complete intersection
              singularity},
  {\em J. Fac. Sci. Univ. Tokyo Sect. IA Math.},
  {Journal of the Faculty of Science. University of Tokyo.
              Section IA. Mathematics},
   {37},
      {1990},
   {No. 1},
      {11--32},
  %    ISSN = {0040-8980},
 %  MRCLASS = {32S40 (32S45 32S55)},
%  MRNUMBER = {1049017},
%MRREVIEWER = {Kurt Behnke},

		
\bibitem{Okabook}
M.~Oka.
\newblock { Non-degenerate complete intersection singularity}.
\newblock {\em Hermann}, Paris, 1997.

\bibitem{OkaPolar}
M.~Oka.
\newblock Topology of polar weighted homogeneous hypersurfaces.
\newblock {\em Kodai Math. J.}, 31,  (2008),  (2):163-182.
\bibitem{OkaMix}
M.~Oka.
\newblock Non-degenerate mixed functions.
\newblock {\em Kodai Math. J.}, 33, (2010),  (1):1-62.
\bibitem{OkaStrong}
M.~Oka.
\newblock Mixed functions of strongly polar weighted homogeneous face type,
  %arxiv 1202.2166v1.
  In {\em Advanced Studies in Pure Math.}, 66, (2015), 173-202.

  \bibitem{OkaAf}
  M.~Oka.
 {On Milnor fibrations of mixed functions, $a_f$-condition and boundary stability}.
{\em Kodai J. Math.}, 38,  (2015), 581-603.

\bibitem{OkaConnectivity}
M.~Oka.
On the connectivity of Milnor fiber for mixed functions,
arXiv: 1809.00545v1.
\bibitem{fg-bar}
M.~Oka.
On the Milnor fibration for $f(\mathbf z)\bar g(\mathbf z)$,
arXiv: 1812.10909v3, to appear in European Journal of Mathematics.
DOI 10.1007/s40879-019-00380-1.
\bibitem{Param-Tibar-revised}
A.J. Parameswaran and M. Tibar. 
Corrigendum to "Thom irregularity and Milnor tube fibrations",
Bull. Sci. Math.,153,  (2019),120-123.
%arXiv:1606.09008v6, 2019.
%\bibitem 
\bibitem{Param-Tibar}
A.J. Parameswaran and M. Tibar. 
Thom irregularity and Milnor tube fibrations,
{\em Bull. Sci. Math.}, 143, (2018), 58-72.
\bibitem{Pichon-Seade0}
 A. Pichon and J. Seade.
  {Real singularities and open-book decompositions of the
              3-sphere},
   {\em Ann. Fac. Sci. Toulouse Math. (6)},
{12},
      (2003),
   {2},
   {245--265}.
\bibitem{Pichon-Seade1}
 A. Pichon and J. Seade.
    {Fibred multilinks and singularities {$f\overline g$}},
    {\em Math. Ann.}, 342,  (2008),
       {3},
 {487-514}.
\bibitem{Pichon-Seade2}
A. Pichon and J. Seade.
     {Milnor fibrations and the {T}hom property for maps
              {$f\overline g$}},
   {\em Journal of Singularities}, 3, (2011),
   %
      {144-150}.
  %    ISSN = {1949-2006},
 %  MRCLASS = {32S55 (32C05 57Q45)},
  \bibitem{Pichon-Seade3}
 A. Pichon and J. Seade.
 \newblock {Erratum: {M}ilnor fibrations and the {T}hom property for maps
              {$f\overline g$}},
   {\em J. Journal of Singularities}, 7, (2013),
     {21--22}.
 \bibitem{R-S-V}
M.~A.~S. Ruas, J.~Seade, and A.~Verjovsky.
\newblock On real singularities with a {M}ilnor fibration.
\newblock In {\em Trends in singularities}, Trends Math., 191-213.
  Birkh\"auser, Basel, 2002.
    \bibitem{SE1}
  J. Seade.
      {On {M}ilnor's fibration theorem for real and complex
              singularities},
in  {\em Singularities in geometry and topology},
   {127--158},
  {World Sci. Publ., Hackensack, NJ},
  {2007}.
  % MRCLASS = {32S55 (57M25)},
 % MRNUMBER = {2311486},
%MRREVIEWER = {Lee Rudolph},
   %    DOI = {10.1142/9789812706812_0004},
  %     URL = {https://doi.org/10.1142/9789812706812_0004},

\bibitem {TibarOberwolfach}
M. ~Tibar.
 Regularity of real mappings and non-isolated singularities, in:
{\em Topology of Real Singularities and Motivic Aspects}. Abstracts from the workshop held 30 September - 6 October, 2012.
{\em Oberwolfach Rep. 9 (2012)}, no. 4, 2933-2934.
\bibitem{WhitneyElementary}
H.~Whitney.
\newblock Elementary structure of real algebraic varieties.
\newblock {\em Ann. of Math. (2)}, (1957), 66:545--556.


\bibitem{WhitneyTangent}
H.~Whitney.
Tangents to an analytic variety.
{\em Ann. of Math.} 81, (1965), 496-549.
\bibitem{Wolf1}
J.~A. Wolf.
\newblock Differentiable fibre spaces and mappings compatible with {R}iemannian
  metrics.
\newblock {\em Michigan Math. J.}, 11, (1964), 65-70,.

\end{thebibliography}
\end{document}